\documentclass[11pt,twoside,english]{article}
\usepackage[latin9]{inputenc}
\usepackage{geometry}
\usepackage{color}
\usepackage{amsthm}
\usepackage{amsmath}
\usepackage{amssymb}
\usepackage{graphicx}
\usepackage{setspace}
\usepackage{esint}
\makeatletter
\numberwithin{equation}{section}
\numberwithin{figure}{section}
\theoremstyle{plain}
\newtheorem{thm}{\protect\theoremname}[section]
  \theoremstyle{definition}
  \newtheorem{defn}[thm]{\protect\definitionname}
  \theoremstyle{remark}
  \newtheorem{rem}[thm]{\protect\remarkname}
  \theoremstyle{plain}
  \newtheorem{cor}[thm]{\protect\corollaryname}
  \theoremstyle{plain}
  \newtheorem*{conjecture*}{\protect\conjecturename}
  \theoremstyle{plain}
  \newtheorem{conjecture}[thm]{\protect\conjecturename}
  \theoremstyle{plain}
  \newtheorem{prop}[thm]{\protect\propositionname}
  \theoremstyle{plain}
  \newtheorem{lem}[thm]{\protect\lemmaname}

\date{}
\pdfpageattr{/Group << /S /Transparency /I true /CS /DeviceRGB>>}
\usepackage{ae,aecompl} 
\usepackage{amssymb}
\usepackage{bbm}
\usepackage{hyperref}

\hypersetup{
                    colorlinks=true,
                    linkcolor=red,
                    citecolor=blue,
                    urlcolor=green,
                    linktocpage=true,
                    pdfstartview=FitH,
                }

\makeatother

\usepackage{babel}
  \providecommand{\conjecturename}{Conjecture}
  \providecommand{\corollaryname}{Corollary}
  \providecommand{\definitionname}{Definition}
  \providecommand{\lemmaname}{Lemma}
  \providecommand{\propositionname}{Proposition}
  \providecommand{\remarkname}{Remark}
\providecommand{\theoremname}{Theorem}

\begin{document}
\global\long\def\R{\mathbb{R}}

\global\long\def\C{\mathbb{C}}

\global\long\def\bbN{\mathbb{N}}

\global\long\def\Z{\mathbb{Z}}

\global\long\def\N{\mathbb{N}}

\global\long\def\Q{\mathbb{Q}}

\global\long\def\T{\mathbb{T}}

\global\long\def\F{\mathbb{F}}

\global\long\def\Sph{\mathbb{S}}

\global\long\def\sub{\subseteq}

\global\long\def\cvx{\mbox{Cvx}\left(\R^{n}\right)}

\global\long\def\cvxo{\text{Cvx}_{0}\left(\R^{n}\right)}

\global\long\def\lcg{\text{LC}_{g}\left(\R^{n}\right)}

\global\long\def\lc{\text{LC}\left(\R^{n}\right)}

\global\long\def\dis{\text{Dis }\left(\R^{n}\right)}

\global\long\def\one{\mathbbm1}

\global\long\def\infc#1#2{#1\,+_{\text{cvx}}\,#2}

\global\long\def\supc#1#2{#1\,+_{\text{lc}}\,#2}

\global\long\def\vol{{\rm Vol}}

\global\long\def\EE{\mathbb{E}}

\global\long\def\cvdold{\,\cdot_{\text{cvx}}\,}

\global\long\def\lcdold{\,\cdot_{\text{lc}}\,}

\global\long\def\cvd{\,\check{\cdot}\,}

\global\long\def\lcd{\,\hat{\cdot}\,}

\global\long\def\cvxdot{\,\cdot_{\text{cvx}}\,}

\global\long\def\cvxplus{\,+_{{\scriptstyle {\rm cvx}}}\,}

\global\long\def\convp{\overset{\smallsmile}{+}}

\global\long\def\logcp{\overset{\smallfrown}{+}}

\global\long\def\cvp{\,+_{\text{cvx}}\,}

\global\long\def\lcp{\,+_{\text{lc}}\,}

\global\long\def\cvone#1{\one_{\left\{  #1\right\}  }^{\infty}}

\global\long\def\cvxone#1{\one_{\left\{  #1\right\}  }^{\infty}}

\global\long\def\epi#1{\text{epi}\left\{  #1\right\}  }

\global\long\def\sp{{\rm sp}}

\global\long\def\K{\mathcal{K}}

\global\long\def\A{\mathcal{A}}

\global\long\def\L{\mathcal{L}}

\global\long\def\P{\mathcal{P}}

\global\long\def\W{\mathcal{W}}

\global\long\def\iprod#1#2{\langle#1,\,#2\rangle}

\global\long\def\cvrs{{\rm Cvrs}}

\global\long\def\cvrsb{\overline{{\rm Cvrs}}}

\global\long\def\uball{B_{2}^{n}}

\global\long\def\conv{{\rm conv}}

\global\long\def\Gauss{\gamma_{n}^{\sigma}}

\global\long\def\EXP#1{\EE_{{\textstyle #1}}}

\global\long\def\eps{\varepsilon}

\global\long\def\PP{\mathbb{P}}

\global\long\def\IndUball{\one_{B_{2}^{n}}}

\global\long\def\J{{\cal J}}

\global\long\def\N{{\cal N}}

\global\long\def\inte#1{{\rm int}\left(#1\right)}

\title{On weighted covering numbers and the Levi-Hadwiger
conjecture}

\author{Shiri Artstein-Avidan and Boaz A. Slomka}

\maketitle

\begin{abstract}
We define new natural variants of the notions of weighted covering
and separation numbers and discuss them in detail. We prove a strong
duality relation between weighted covering and separation numbers
and prove a few relations between the classical and weighted covering
numbers, some of which hold true without convexity assumptions and
for general metric spaces. As a consequence, together with some volume
bounds that we discuss, we provide a  bound for the famous Levi-Hadwiger
problem concerning covering a convex body by homothetic slightly smaller
copies of itself, in the case of centrally symmetric convex bodies,
which is qualitatively the same as the best currently known bound.
We also introduce the weighted notion of the Levi-Hadwiger covering
problem, and settle the centrally-symmetric case, thus also confirm
the equivalent fractional illumination conjecture \cite[Conjecture 7]{Naszodi09}
in the case of centrally symmetric convex bodies (including the characterization
of the equality case, which was unknown so far).
\end{abstract}

\section{Introduction}

\subsection{Background and Motivation}

Covering numbers can be found in various fields of mathematics, including
combinatorics, probability, analysis and geometry. They often participate
in the solution of many problems in quite a natural manner.

In the combinatorial world, the idea of fractional covering numbers
is well-known and utilized for many years. In \cite{ArtsteinRaz2011},
the authors introduced the weighted notions of covering and separation
numbers of convex bodies and shed new light on the relations between
the classical notions of covering and separation, as well as on the
relations between the classical and weighted notions. In this note
we propose a variant of these numbers which is perhaps more natural
and discuss these numbers in more detail, revealing more useful relations,
as well as some applications. To state our main results, we need some
definitions. The impatient reader may skip the following section and
go directly to Section \ref{sec:MainResults} where the main results
are stated. 

Apart from deepening our understanding of these notions, and revealing
more useful relations, we also consider this work as a first step
towards the functionalization of covering and separation numbers;
in the past decade, various parts from the theory of convex geometry
have been gradually extended to the realm of log-concave functions.
Numerous results found their functional generalizations. One natural
way to embed convex sets in $\R^{n}$ into the class of log-concave
functions is to identify every convex set $K$ with its characteristic
function $\one_{K}$. Besides being independently interesting, such
extensions may sometimes be applied back to the setting of convex
bodies. For further reading, we refer the reader to \cite{AKM2005,Ball88,Barth98,Klartag2007,Klartag2005}.
Since covering numbers play a considerable part in the theory of convex
geometry, their extension to the realm of log-concave functions seems
to be an essential building block for this theory. Our results using
functional covering numbers will be published elsewhere.

\subsection{Definitions}

Let $K\sub\R^{n}$ be compact and let $T\sub\R^{n}$ be compact with
non-empty interior. The classical covering number of $K$ by $T$
is defined to be the minimal number of translates of $T$ such that
their union covers $K$, namely
\[
N\left(K,T\right)=\min\left\{ N\,:\, N\in\bbN,\,\,\exists x_{1},\dots x_{N}\in\R^{n};\,\, K\sub\bigcup_{i=1}^{N}\left(x_{i}+T\right)\right\} .
\]
Here and in the sequel we assume that the covered set $K$ is compact
and the covering set $T$ has non-empty interior so that the covering
number will be finite. However, one may remove these restriction so
long as we are content also with infinite outcomes. 

A well-known variant of the covering number is obtained by considering
only translates of $T$ that are centered in $K$, namely
\[
\overline{N}\left(K,T\right)=\min\left\{ N\,:\, N\in\bbN,\,\,\exists x_{1},\dots x_{N}\in K;\,\, K\sub\bigcup_{i=1}^{N}\left(x_{i}+T\right)\right\} .
\]
Clearly, $N\left(K,T\right)\le\overline{N}\left(K,T\right)$, and
it is easy to check that for convex bodies%
\footnote{by convex body we mean, here and in the sequel, a compact convex set
with non-empty interior%
} $K$ and $T$, we have $\overline{N}\left(K,T-T\right)\le N\left(K,T\right)$.
Furthermore, if $T$ is a Euclidean ball then $N\left(K,T\right)=\overline{N}\left(K,T\right)$.

The classical notion of the separation number of $T$ in $K$ is closely
related to covering numbers and is defined to be the maximal number
of non-overlapping translates of $T$ which are centered in $K$;
\[
M\left(K,T\right)=\max\left\{ M\,:\, N\in\bbN,\,\,\exists x_{1},\dots x_{M}\in K\,;\,\,\left(x_{i}+T\right)\cap\left(x_{j}+T\right)=\emptyset\,\,\forall i\neq j\right\} .
\]
It is a standard equivalence relation that $N\left(K,T-T\right)\le M\left(K,T\right)\le N\left(K,T\right)$.
We also define the less conventional 
\[
\overline{M}\left(K,T\right)=\max\left\{ M\,:\, N\in\bbN,\,\,\exists x_{1},\dots x_{M}\in K\,;\,\,\left(x_{i}+T\right)\cap\left(x_{j}+T\right)\cap K=\emptyset\,\,\forall i\neq j\right\} .
\]
Note that the condition $\left(x_{i}+T\right)\cap\left(x_{j}+T\right)=\emptyset$
is equivalent to $x_{i}-x_{j}\not\in T-T$ which means that $M\left(K,T\right)=M\left(K,-T\right)=M\left(K,\frac{T-T}{2}\right)$.
Moreover, it is easily checked that for a convex $K$ and for a centrally
symmetric convex body $L$ (i.e., $L=-L$) we have $M\left(K,L\right)=\overline{M}\left(K,L\right)$
and thus by the last remark $M\left(K,T\right)=\overline{M}\left(K,T\right)$
for any convex body $T$. In the sequel, we will define weighted counterparts
for $M\left(K,T\right)$ and $\overline{M}\left(K,T\right)$ which
will not necessarily be equal, even in the convex and centrally symmetric
case.  \\

In order to define the weighted versions, let $\one_{A}$ denote the
indicator function of a set$A\sub\R^{n}$, equal to $1$ if $x\in A$
and $0$ if $x\not\in A$.
\begin{defn}
A sequence of pairs $S=\{(x_{i}\,,\,\omega_{i}):\, x_{i}\in\R^{n},\,\omega_{i}\in\R^{+}\}_{i=1}^{N}$
of points and weights is said to be a weighted covering of $K$ by
$T$ if for all $x\in K$ we have $\sum_{i=1}^{N}\omega_{i}\one_{x_{i}+T}\left(x\right)\ge1$.
The total weight of the covering is denoted by $\omega(S)=\sum_{i=1}^{N}\omega_{i}$.
The weighted covering number of $K$ by $T$ is defined to be the
infimal total weight over all weighted coverings of $K$ by $T$ and
is denoted by $N_{\omega}\left(K,T\right)$.
\end{defn}

\noindent One may consider only coverings $S=\{(x_{i}\,,\,\omega_{i}):\, x_{i}\in K,\,\omega_{i}\in\R^{+}\}_{i=1}^{N}$
with centers of $T$ in $K$. The corresponding weighted covering
number for such coverings, denoted here by $\overline{N}_{\omega}\left(K,T\right)$
is defined to be the infimal total weight over such coverings. Clearly,
$\overline{N}_{\omega}(K,T)\le N_{\omega}(K,T)$. The weighted notions
of covering and separation numbers corresponding to $\overline{N}\left(K,T\right)$
and $\overline{M}\left(K,T\right)$ were introduced in \cite{ArtsteinRaz2011}.
In this note, we shall focus on the weighted versions of $N\left(K,T\right)$
and $M\left(K,T\right)$. \\

Let us reformulate the above definitions in the language of measures.
Note that the covering condition $\sum_{i=1}^{N}\omega_{i}\one_{x_{i}+T}\left(x\right)\ge1$
for all $x\in K$ is equivalent to $\nu*\one_{T}\ge\one_{K}$ where
$\nu=\sum_{i=1}^{N}\omega_{i}\delta_{x_{i}}$ is the discrete measure
with masses $\omega_{i}$ centered at $x_{i}$ and where $*$ stands
for the convolution
\[
\left(\nu*\one_{T}\right)\left(x\right)=\int_{\R^{n}}\one_{T}\left(x-y\right)d\nu\left(y\right).
\]
Let ${\cal D}_{+}^{n}$ denote all non-negative discrete and finite
measures on $\R^{n}$ and let ${\rm supp}\left(\nu\right)\sub\R^{n}$
denote the support of a measure $\nu$ on $\R^{n}$. Thus, the weighted
covering numbers of $K$ by $T$ can be written as
\[
N_{\omega}\left(K,T\right)=\inf\left\{ \nu(\R^{n})\,:\,\nu*\one_{T}\ge\one_{K}\,,\nu\in{\cal D}_{+}^{n}\right\} 
\]
and
\[
\overline{N}_{\omega}\left(K,T\right)=\inf\left\{ \nu(\R^{n})\,:\,\nu*\one_{T}\ge\one_{K}\,,\nu\in{\cal D}_{+}^{n}\,\,{\rm with}\,\,{\rm supp}\left(\nu\right)\sub K\right\} .
\]

It is this natural to extend this notion of covering to general non-negative
measures. Let ${\cal B}_{+}^{n}$ denote all non-negative Borel regular
measures on $\R^{n}$. 
\begin{defn}
\noindent Let $K\sub\R^{n}$ be compact and let $T\subset\R^{n}$
be compact with non-empty interior. A non-negative measure $\mu\in{\cal B}_{+}^{n}$
is said to be a covering measure of $K$ by $T$ if $\mu*\one_{T}\ge\one_{K}$.
The corresponding weighted covering number is defined by
\[
N^{*}\left(K,T\right)=\inf\left\{ \int_{\R^{n}}d\mu\,:\,\mu*\one_{T}\ge\one_{K}\,,\mu\in{\cal B}_{+}^{n}\right\} .
\]
Clearly, $N^{*}\left(K,T\right)\le N_{\omega}\left(K,T\right)$. In
Proposition \ref{Prop:OptimalCov}, we show that the above infimum
is actually a minimum, that is, there exists an optimal covering Borel
measure of $K$ by $T$. Note that the set of optimal covering measures
is clearly convex. 
\end{defn}
The weighted notions of the separation are defined similarly; a measure
$\mu\in{\cal B}_{+}^{n}$ is said to be $T-$separated if $\mu*\one_{T}\le1$.
The weighted separation numbers, corresponding to $N_{\omega}\left(K,T\right)$,
$\overline{N}_{\omega}\left(K,T\right)$ and $N^{*}\left(K,T\right)$
are respectively defined by:

\[
M_{\omega}\left(K,T\right)=\sup\left\{ \int_{K}d\nu\,:\,\nu*\one_{T}\le1,\,\,\nu\in{\cal D}_{+}^{n}\right\} ,
\]
\[
\overline{M}_{\omega}\left(K,T\right)=\sup\left\{ \int_{K}d\nu\,:\,\,\forall x\in K\,\,\left(\nu*\one_{T}\right)\left(x\right)\le1,\,\,\nu\in{\cal D}_{+}^{n}\right\} 
\]
and
\[
M^{*}\left(K,T\right)=\sup\left\{ \int_{K}d\mu\,:\,\mu*\one_{T}\le1\,,\mu\in{\cal B}_{+}^{n}\right\} ,
\]
where again clearly $M_{\omega}\left(K,T\right)\le M^{*}\left(K,T\right)$.

\subsection{Main Results \label{sec:MainResults}}

Our first main result is a strong duality between weighted covering
and separation numbers; it turns out that $N^{*}\left(K,T\right)$
and $M^{*}\left(K,-T\right)$ can be interpreted as the outcome of
two dual problems in the sense of linear programming. Indeed, as in
\cite{ArtsteinRaz2011}, this observation is a key ingredient in the
proof of our first main result below which states that the outcome
of these dual problems is the same (we call this ``strong duality'').

\begin{thm}
\label{thm:StrongDuality}Let $K\sub\R^{n}$ be compact and let $T\sub\R^{n}$
be a compact with non-empty interior. Then 
\[
M_{\omega}\left(K,T\right)=M^{*}\left(K,T\right)=N^{*}\left(K,-T\right).
\]
\end{thm}
\begin{rem}
While it is not clear, so far, whether strong duality also holds for fractional
covering numbers with respect to discrete measures, namely whether
$N_{\omega} (K,T )=M_{\omega} (K,-T )$, one may show that 
\[\lim_{\delta \to 0^+}N_{\omega} (K,-(1+\delta)T ) = \lim_{\delta\to 0} M_{\omega} (K,(1+\delta)T )\le M_{\omega}(K,T). \] 
In particular, for almost every $t>0$
\[
M_{\omega} (K,tT )=M^{*} (K,tT )=N^{*} (K,-tT )=N_{\omega} (K,tT ).
\]
See discussion in Section \ref{sub:Strong-duality}, Remark \ref{rem:Orit-paper}.
\end{rem}
As a consequence of Theorem \ref{thm:StrongDuality}, together with
the well-known homothety equivalence between classical covering and
separation numbers $N\left(K,T-T\right)\le M\left(K,T\right)\le N\left(K,T\right)$,
we immediately get the following equivalence relation between the
classical and weighted covering numbers (which has also appeared in
\cite{ArtsteinRaz2011} for the pair $\overline{M}_{\omega},\overline{N}_{\omega}$). 
\begin{cor}
\label{cor:HomothetyEq.}
Let $K\sub\R^{n}$ be compact and let $T\sub\R^{n}$
be compact with non-empty interior. Then 
Then
\[
N\left(K,T-T\right)\le N_{\omega}\left(K,T\right)\le N\left(K,T\right)
\]
\end{cor}
\noindent We remark that Corollary \ref{cor:HomothetyEq.} is actually
implied by the weak duality $M^{*}\left(K,-T\right)\le N^{*}\left(K,T\right)$
which we prove in Proposition \ref{prop:WeakDuality} below, the proof
of which is relatively simple. Similarly, we shall prove in Proposition
\ref{prop:WeakDuality} that $\overline{M}_{\omega}\left(K,-T\right)\le\overline{N}_{\omega}\left(K,T\right)$
providing an alternative short proof for the weak duality result in
\cite[Theorem 6]{ArtsteinRaz2011}.\\

For a centrally symmetric convex set $T$, Corollary \ref{cor:HomothetyEq.}
reads $N\left(K,2T\right)\le N_{\omega}\left(K,T\right)\le N\left(K,T\right)$.
Although this ``constant homothety'' equivalence of classical and
weighted covering is useful, it turns out to be insufficient in certain
situations. To that end, we introduce our second main result, in which
the homothety factor $2$ is replaced by a factor $1+\delta$ with
$\delta>0$ arbitrarily close to $0$. This gain is diminished by
an additional logarithmic factor; such a result is a reminiscent of
Lov�sz's \cite{Lovasz1975} well-known inequality for fractional covering
numbers of hypergraphs.
\begin{thm}
\label{thm:DeltaEquiv}Let $K\sub\R^{n}$ be compact and let $T_{1},T_{2}\sub\R^{n}$
be compact with non-empty interior. Then 
\[
N\left(K,T_{1}+T_{2}\right)\le\ln\left(4\overline{N}\left(K,T_{2}\right)\right)\left(N_{\omega}\left(K,T_{1}\right)+1\right)+\sqrt{\ln\left(4\overline{N}\left(K,T_{2}\right)\right)\left(N_{\omega}\left(K,T_{1}\right)+1\right)}
\]

\end{thm}
We remark that for the proof of our application in Section \ref{sec:Hadwiger}
below, we shall use $T_{1}=\delta T$ and $T_{2}=\left(1-\delta\right)T$
for $0<\delta<1$ and a single convex body $T$. It is also worth
mentioning that Theorem \ref{thm:DeltaEquiv} holds for $\overline{N}_{\omega}\left(K,T\right)$
and $\overline{N}\left(K,T\right)$ as well (with the exact same proof).

\subsection{Additional inequalities}

Let $\vol\left(A\right)$ denote the Lebesgue volume of a set $A\sub\R^{n}$.
The classical covering and separation numbers satisfy simple volume
bounds. Such volume bounds also hold for the weighted case, and  turn
out to be quite useful.
\begin{thm}
\label{thm:VolumeBounds}Let $K\sub\R^{n}$ be compact and let $T\sub\R^{n}$
be compact with non-empty interior. Then
\[
\max\left\{ \frac{{\rm Vol} (K )}{{\rm Vol} (T )},1\right\} \le N^{*} (K,T )\le\frac{{\rm Vol} (K-T )}{{\rm Vol} (T )}.
\]

\end{thm}

\begin{rem}
\label{rem:MneqN}Let us show, by using the above volume bounds, that
classical and weighted covering numbers are not equal in general,
even for centrally symmetric convex bodies such as a cube and a ball
(for a simple $2-$dimensional example, see the last part of Remark
\ref{rem:OptimalCovSep}). Namely, we show that $N_{\omega} (K,T )\neq N (K,T )$
where $T=B_{2}^{n}$ is the unit ball in $\R^{n}$ and $K=[-R,R]^{n}$
for a large enough $R$. Indeed, it was shown in \cite{ErdosRogers53}
that the lower limit of the density of covering a cube by balls, defined
as the limit of the ratio $N ([-R,R]^{n},B_{2}^{n} )\cdot\vol (B_{2}^{n} )/(2R)^{n}$,
as $R$ tends to infinity is bounded from below by $16/15-\eps_{n}$
where $\eps_{n}\to0$ as $n\to\infty$. However, by our volume bounds
in Theorem \ref{thm:VolumeBounds}, it follows that the weighted covering
density $N_{\omega} ([-R,R]^{n},B_{2}^{n} )\cdot\vol (B_{2}^{n} )/(2R)^{n}$
approaches $1$ as $R\to\infty$. Note that by Proposition \ref{prop:WeakDuality}
below, this also means that $M (Q,B_{2}^{n} )\neq N(Q,B_{2}^{n})$
for a large enough cube and dimension.
\end{rem}

\subsection{An application}

A famous conjecture, known as the Levi-Hadwiger or the Gohberg-Markus
covering problem, was posed in \cite{Levi55}, \cite{Hadwiger57}
and \cite{GohMar60}. It states that in order to cover a convex set
by slightly smaller copies of itself, one needs at most $2^{n}$ copies. 
\begin{conjecture*}
Let $K\sub\R^{n}$ be a convex body with non empty interior. Then
there exists $0<\lambda<1$ such that 
\[
N(K,\lambda K)\le2^{n}.
\]
Equivalently, $N (K,{\rm int} (K ) )\le 2^{n}$.
Moreover, equality holds if and only if $K$ is a parallelotope. 
\end{conjecture*}
This problem has drawn much attention over the years, but only little
has been unraveled so far. We mention that Levi confirmed the conjecture
for the plane, and that Lassak confirmed it for centrally symmetric
bodies in $\R^{3}$. The currently best known general upper bound
for $n\ge3$ is ${2n \choose n}\left(n\ln n+n\ln\ln n+5n\right)$
and the best bound for centrally symmetric convex bodies is $2^{n}$$\left(n\ln n+n\ln\ln n+5n\right)$,
both of which are simple consequences of Rogers' bound for the asymptotic
lower densities for covering the whole space by translates of a general
convex body, see \cite{RogersZong}. For a comprehensive survey of this problem and the aforementioned
results see \cite{BrassMoser05}.

It is natural, after introducing weighted covering, to formulate the
Levi-Hadwiger covering problem for the case of weighted covering. 
\begin{conjecture}
\label{conj:WeightedHadwiger}Let $K\sub\R^{n}$ be a convex body.
Then ${\displaystyle \lim_{\lambda\to1^{-}}N_{\omega}\left(K,\lambda K\right)\le2^{n}.}$
Moreover, equality holds if and only if $K$ is a parallelotope. 
\end{conjecture}
For centrally symmetric convex bodies, we verify Conjecture \ref{conj:WeightedHadwiger},
including the equality case. We show
\begin{thm}
\label{thm:WeightedHadwiger}Let $K\sub\R^{n}$ be a convex body.
Then
\[
\lim_{\lambda\to1^{-}}N_{\omega}\left(K,\lambda K\right)\le\begin{cases}
2^{n} & K=-K\\
{2n \choose n} & K\neq-K
\end{cases}
\]
Moreover, for centrally symmetric $K$, ${\displaystyle \lim_{\lambda\to1^{-}}N_{\omega}\left(K,\lambda K\right)=2^{n}}$
if and only if $K$ is a parallelotope.
\end{thm}
It is worth mentioning that the classical covering problem of Levi-Hadwiger
is equivalent to the problem of the illumination of a convex body
(for surveys see \cite{MartiniEt99,Bezdek06}) which asks how many
directions are required to illuminate the entire boundary of a convex
body $K$ (a direction $u\in\Sph^{n-1}$ is said to illuminate a point
$b$ in the boundary of $K$ if the ray emanating from $b$ in direction
$u$ intersects the interior of $K$). A fractional version of the
illumination problem was considered in \cite{Naszodi09}, where it
was proven that the fractional illumination number of a convex body
$K$, denoted by $i^{*}\left(K\right)$, satisfies that $i^{*}\left(K\right)\le{2n \choose n}$
and that $i^{*}\left(K\right)\le2^{n}$ for all centrally symmetric
bodies (with parallelotopes attaining equality). It was further conjectured
\cite[Conjecture 7]{Naszodi09} that $i^{*}\left(K\right)\le2^{n}$
for all convex bodies and that equality is attained only for parallelotopes.
 However, as no relation between fractional and usual illumination
numbers was proposed, this result remained isolated. Also, it seems
that the equality conditions were not analyzed. In fact, one may
verify that the proof of the equivalence between  the illumination
problem and the Levi-Hadwiger covering problem (see \cite[Theorem 7]{BolGoh85})
carries over to the fractional setting and conclude that  $i^{*}\left(K\right)=\lim_{\lambda\to1^{-}}N_{\omega}\left(K,\lambda K\right)$.
Thus, Theorem \ref{thm:WeightedHadwiger} actually verifies the aforementioned
results about fractional illumination and also verifies \cite[Conjecture 7]{Naszodi09}
for the case of centrally symmetric convex bodies, including the equality
hypothesis.\\

Combining the inequality in Theorem \ref{thm:DeltaEquiv} with the
volume inequality in Theorem \ref{thm:VolumeBounds}, we prove the
following bound for the classical Levi-Hadwiger problem, in the case
of centrally symmetric convex bodies, which is the same as the aforementioned
(best known) general bound of Rogers.
\begin{cor}
\label{cor:Hadwiger} Let $K\sub\R^{n}$ be a centrally symmetric
convex body. Then for all $n\ge3$, 
\begin{align*}
\lim_{\lambda\to1^{-}}N\left(K,\lambda K\right) & \le2^{n}\left(n\ln\left(n\right)+n\ln\ln\left(n\right)+5n\right)
\end{align*}

\end{cor}
\noindent We remark that the above bound and Rogers' bound are asymptotically
equivalent, and that in both cases the constant $5n$ above may be
improved by performing more careful computations, improving and optimizing
over various constants. We avoid such computations as they will not
affect the order of magnitude of this bound, and complicate the exposition.

\subsection*{Acknowledgments}

\textcolor{black}{We thank Prof.~Noga~Alon, Prof.~Mark~Meckes and
Prof.~Boris~Tsirelson for their valuable comments and suggestions.
We also thank Prof.~Rolf~Schneider for his proof of Lemma \ref{lem:CcapC}
and for translating for us the entire paper \cite{DanzerGrunbaum}
from German.}\\
This research was supported in part by ISF grant number  247/11.

The remainder of this note is organized as follows. In Section \ref{sub:Weak-duality}
we show weak duality between weighted covering and separation numbers.
In Section \ref{sub:Strong-duality} we prove Theorem \ref{thm:StrongDuality}.
In Section \ref{sub:Optimal-measures} we discuss the existence of
optimal covering measures. In Section \ref{sec:GCC} we discuss the
approximation of uniform covering measures by discrete covering measures.
In  \ref{sub:Volume-bounds} we prove Theorem \ref{thm:VolumeBounds}.
In Section \ref{sub:delta-equivalence} we prove Theorem \ref{thm:DeltaEquiv}.
In Section  \ref{sec:The-metric-space-setting} we discuss the weighted
notions of covering and separation in the setting of general metric
spaces. In Section \ref{sec:Hadwiger} we discuss both the classical
and weighted versions of the Levi-Hadwiger covering problems, proving
Theorem \ref{thm:WeightedHadwiger} and Corollary \ref{cor:Hadwiger}.

\section{Weighted covering and separation}

\subsection{\label{sub:Weak-duality}Weak duality }
\begin{prop}
\label{prop:WeakDuality}Let $K\sub\R^{n}$ be compact and let $T\sub\R^{n}$
be compact with non-empty interior. Then 
\[
M^{*} (K,T )\le N^{*} (K,-T )\,\;{\rm and}\,\;\overline{M}_{\omega} (K,T )\le\overline{N}_{\omega} (K,-T )
\]
In particular, we also have that $M_{\omega} (K,T )\le N_{\omega} (K-T ).$\end{prop}
\begin{proof}
Let $\mu$ be a covering measure of $K$ by $-T$. Let $\rho$ be
a $T-$separated measure. By our assumptions we have that $\one_{T}*\rho\le1$
and $\one_{-T}*\mu\ge\one_{K}$. Thus 
\begin{align*}
\int_{K}d\rho\left(x\right)= & \int\one_{K}\left(x\right)\cdot d\rho\left(x\right)\le\int\left(\one_{-T}*\mu\right)\left(x\right)d\rho\left(x\right)\\
= & \int d\rho\left(x\right)\int\one_{-T}\left(x-y\right)d\mu\left(y\right)\\
= & \int d\mu\left(y\right)\int\one_{T}\left(y-x\right)d\rho\left(x\right)\\
= & \int\left(\one_{T}*\rho\right)\left(y\right)d\mu\left(y\right)\\
\le & \int d\mu\left(y\right)
\end{align*}
and so $M^{*}\left(K,T\right)\le N^{*}\left(K,-T\right)$. Similarly,
by considering $\one_{T}*\rho\le1$ only on $K$ and $\mu$ which
must be supported only on $K$, the exact same inequality yields $\overline{M}_{\omega}\left(K,T\right)\le\overline{N}_{\omega}\left(K,-T\right)$.
\end{proof}

\subsection{\label{sub:Strong-duality}Strong duality}

In this section we prove Theorem \ref{thm:StrongDuality}. By Proposition
\ref{prop:WeakDuality} it is enough to show an inequality $M_{\omega}\left(K,T\right)\ge N_{\omega}\left(K,-T\right)$.

We start with the discretized versions of our weighted covering and
separation notions. Let $\Lambda=\left\{ x_{i}\right\} _{i=1}^{d}\sub\R^{n}$
be some finite set, which will be chosen later, and define:
\[
N_{\omega}\left(K,T,\Lambda\right)=\inf\left\{ \sum_{i=1}^{N}\omega_{i}\,:\,\,\exists\left(x_{i},\,\omega_{i}\right)_{i=1}^{N}\sub\left(\Lambda,\,\R^{+}\right),\,\,\sum_{i=1}^{N}\omega_{i}\one_{T}\left(x-x_{i}\right)\ge1_{K}(x)\,\:\forall x\in\Lambda\right\} 
\]
and
\[
M_{\omega}\left(K,T,\Lambda\right)=\sup\left\{ \sum_{i=1}^{N}\omega_{i}\,:\,\,\exists\left(x_{i},\,\omega_{i}\right)_{i=1}^{N}\sub\left(\Lambda\cap K,\,\R^{+}\right),\,\,\sum_{i=1}^{N}\omega_{i}\one_{T}\left(x-x_{i}\right)\le1\,\:\forall x\in\Lambda\right\} .
\]
In this setting, linear programming duality gives us an equality of
the form 
\begin{equation}
N_{\omega}\left(K,T,\Lambda\right)=M_{\omega}\left(K,-T,\Lambda\right).\label{eq:DiscretizeDuality}
\end{equation}
Indeed, define the vectors $b,c\in\R^{d}$ by
\[
c_{i}=\begin{cases}
1, & x_{i}\in K\\
0, & {\rm otherwise}
\end{cases}\,\,,\,\,\,\,\,\,\,\, b_{i}=1
\]
and the $d\times d$ matrix $M$ by 
\[
M_{ij}=\begin{cases}
1, & x_{i}\in x_{j}+T\\
0, & {\rm otherwise.}
\end{cases}
\]
Note that 
\[
M_{ij}^{T}=\begin{cases}
1, & x_{i}\in x_{j}-T\\
0, & {\rm otherwise.}
\end{cases}
\]
Let $\iprod{\cdot}{\cdot}$ denote the standard Euclidean inner product
in $\R^{d}$. Then, in the language of vectors and matrices, the above
discretized weighted covering and separation notions read 
\begin{eqnarray*}
N_{\omega}\left(K,T,\Lambda\right) & = & \min\left\{ \iprod bx:\, Mx\ge c,\, x\ge0\right\} ,\\
M_{\omega}\left(K,-T,\Lambda\right) & = & \max\left\{ \iprod cy:\, M^{T}y\le b,\, y\ge0\right\} 
\end{eqnarray*}
which are equal by the well-known duality theorem of linear programming,
see e.g., \cite{Barvinok2002}.\\

Next, we shall use this observation with a specific family of sets
$\Lambda(\delta).$ A set $\Lambda\left(\delta\right)\sub\R^{n}$
is said to be a $\delta-$net of a set $A\sub\R^{n}$ if for every
$x\in A$ there exists $y\in\Lambda\left(\delta\right)$ for which
$\left|x-y\right|\le\delta$. In other words, $A\sub\Lambda+\delta B_{2}^{n}$.
We shall make use of the two following simple lemmas, corresponding
to \cite[Lemmas 14-15]{ArtsteinRaz2011}. 
\begin{lem}
\label{lem:14Art}Let $K\sub\R^{n}$ be compact, $T\sub\R^{n}$ compact
with non-empty interior and let $\Lambda\left(\delta\right)\sub K$
be some $\delta-$net for $K$. Then 
\[
N_{\omega}\left(K,T+\delta B_{2}^{n}\right)\le N_{\omega}\left(K,T,\Lambda\left(\delta\right)\right).
\]
\end{lem}
\begin{proof}
Indeed, we have that 
\[
N_{\omega}\left(K,T+\delta B_{2}^{n}\right)\le N_{\omega}\left(K\cap\Lambda\left(\delta\right)+\delta B_{2}^{n},T+\delta B_{2}^{n}\right)\le N_{\omega}\left(K\cap\Lambda\left(\delta\right),T\right)\le N_{\omega}\left(K,T,\Lambda\left(\delta\right)\right).
\]
\end{proof}
\begin{lem}
\label{lem:15Art}Let $K\sub\R^{n}$ be compact, $T\sub\R^{n}$ be
compact with non-empty interior and let $\Lambda\left(\delta\right)\sub\R^{n}$
be some $\delta-$net for $K+T$. Then
\[
M_{\omega}\left(K,T\right)\ge M_{\omega}\left(K,T+\delta B_{2}^{n},\Lambda\left(\delta\right)\right)
\]
\end{lem}
\begin{proof}
Suppose that $\left\{ \left(x_{i},\omega_{i}\right)\right\} _{i=1}^{M}\sub\left(K\cap\Lambda\left(\delta\right),\R^{+}\right)$
satisfies the condition in the definition of $M_{\omega}\left(K,T+\delta B_{2}^{n},\Lambda\right)$,
namely for all $x\in\Lambda(\delta)$ we have that $\sum_{i=1}^{N}\omega_{i}\one_{T+\delta B_{2}^{n}}\left(x-x_{i}\right)\le1$.
Then it is also weighted $T-$separated in the usual sense (that is,
satisfying for all $x\in\R^{n}$ that $\sum_{i=1}^{N}\omega_{i}\one_{T}\left(x-x_{i}\right)\le1$).
Indeed, otherwise we would have a point in $x\in\R^{n}$ such that
$\sum_{i=1}^{M}\omega_{i}\one_{T}\left(x-x_{i}\right)>1$. Since $x_{i}\in K$,
it follows that $x\in K+T$ and so there exists a point $y\in\Lambda\left(\delta\right)$
for which $y-x\in\delta B_{2}^{n}$ which means that $\sum_{i=1}^{M}\omega_{i}\one_{T+\delta B_{2}^{n}}\left(y-x_{i}\right)>1$,
a contradiction to our assumption.
\end{proof}
Finally, to prove Theorem \ref{thm:StrongDuality} we shall need the
following continuity result for weighted covering numbers:
\begin{prop}
\label{prop:ContCoverDisc}Let $K\sub\R^{n}$ be compact and let $T\sub\R^{n}$
be compact with non-empty interior. Then
\[
\lim_{\delta\to 0^{+}}N^{*}(K,T+\delta B_{2}^{n})=N^{*}(K,T).
\]
\end{prop}
\begin{proof}
Clearly we have that 
\[
{\displaystyle \lim_{\delta\to 0}N^{*}\left(K,T+\delta B_{2}^{n}\right)\le N^{*}\left(K,T\right)}.
\]
For the opposite direction, let $\delta_{k}\underset{k\to\infty}{\longrightarrow}0$
and let $f_{k}$ be a sequence of continuous functions satisfying
$\one_{T}\le f_{k}\le\one_{T+\delta_{k}D}$ so that $f_{k}\underset{k\to\infty}{\longrightarrow}\one_{T}$
point-wise monotonically. Let $\left(\mu_{k}\right)_{k\in\bbN}$ be
a sequence of covering Borel regular measures of $K$ by $f_{k}$
(the definition is straightforward: replace $\one_{T}$ in the original
definition by $f_{k}$) such that $\int_{\R^{n}}d\mu_{k}\left(x\right)=N^{*}\left(K,f_{k}\right)+\eps_{k}$
with $0<\eps_{k}\to0$. By the well-known Banach-Alaoglu theorem and
passing to a subsequence we may assume without loss of generality
that $\mu_{k}\overset{w^{*}}{\longrightarrow}\mu$ for some non-negative
regular Borel measure. We claim that $\mu$ is a covering measure
of $K$ by $T$. Indeed, let $x\in K$. For $k\ge l$ we have that
\[
1\le\left(\mu_{k}*f_{k}\right)\left(x\right)\le\left(\mu_{k}*f_{l}\right)\left(x\right).
\]
By the weak{*} convergence of $\mu_{k}$ to $\mu$, taking the limit
$k\to\infty$ implies that $1\le\left(\mu*f_{l}\right)\left(x\right)$
and hence, by the monotone convergence theorem, taking the limit $l\to\infty$
implies that $1\le\left(\mu*\one_{T}\right)\left(x\right)$. Thus,
$\mu$ is a covering measure of $K$ by $T$. This means that 
\[
\lim_{k\to\infty}N^{*}\left(K,f_{k}\right)=\lim_{k\to\infty}\int_{\R^{n}}d\mu_{k}=\int_{\R^{n}}d\mu\ge N^{*}\left(K,T\right)
\]
which in turn implies the equality ${\displaystyle \lim_{\delta\to0^{+}}N^{*}\left(K,T+\delta B_{2}^{n}\right)=N^{*}\left(K,T\right)}$,
as claimed. 
\end{proof}

\begin{proof}
[Proof of Theorem \ref{thm:StrongDuality}]We use lemmas \ref{lem:14Art}-\ref{lem:15Art}
together with \eqref{eq:DiscretizeDuality} as follows; let $\Lambda(\delta_{k})$
be a sequence of $\delta_{k}$-nets for $K+T$ with $\delta_{k}\to0^{+}$
such that $K\cap\Lambda\left(\delta_{k}\right)$ are $\delta_{k}-$nets
for $K$. For each $k$ we have 
\begin{equation}
\begin{split}M_{\omega}(K,T) & \geq M_{\omega}(K,T+\delta_{k}B_{2}^{n},\Lambda(\delta_{k}))\\
 & =N_{\omega}(K,-\left(T+\delta_{k}B_{2}^{n}\right),\Lambda(\delta_{k}))\\
 & \geq N_{\omega}(K,-\left(T+2\delta_{k}B_{2}^{n}\right)).
\end{split}
\end{equation}
Thus, by Proposition \ref{prop:ContCoverDisc} 
\[
M_{\omega}(K,T)\ge\lim_{k\to\infty}N^{*}(K,-T+2\delta_{k}B_{2}^{n})=N^{*}\left(K,-T\right).
\]
Taking the above inequality into account together with Proposition
\ref{prop:WeakDuality}, the proof is thus complete.
\end{proof}

\begin{rem}\label{rem:Orit-paper}
In \cite{ArtsteinRaz2011}, Proposition 22 is analogous to Proposition  \ref{prop:ContCoverDisc} above with $N_\omega$ instead of $N^*$. We mention that replacing $T+\delta B_2^n$ by $(1+\delta)T$ is of no significance because any two bodies in fixed dimension are equivalent. 
The proof presented in \cite{ArtsteinRaz2011} is not correct, as it is based on  \cite[Lemma 20]{ArtsteinRaz2011}  which contains an error. \\
Note, however, that since the function $N^{*}(K,tT)$ is monotone in $t>0$, it is clearly continuous almost everywhere. This, combined with the reasoning in \cite[Proof of Theorem 7]{ArtsteinRaz2011} (or, similarly, the reasoning above for $N^*$) implies that for almost every $t>0$ we have
\[
M_{\omega} (K,tT )=N_{\omega} (K,-tT ).
\]
By taking the limit as $t\to 1^+$ we get that 
 \[
 \lim_{\delta\to 0^+}M_{\omega} (K,(1+\delta)T )=\lim_{\delta\to 0^+}N_{\omega} (K,-(1+\delta)T ),
 \]
which, combined with Theorem \ref{thm:StrongDuality}), we get the following row of equalities (as $N^*$, and so also $M_\omega$, are continuous), holding for all convex bodies $K,T\subset \R^n$
\[
M_{\omega} (K,T )=M^{*} (K,T )=N^{*} (K,-T )= \lim_{\delta\to 0^+}N_{\omega} (K,-(1+\delta)T ).
\]
\end{rem}

\subsection{\label{sub:Optimal-measures}Optimal measures }
\begin{prop}
\label{Prop:OptimalCov}Let $K\sub\R^{n}$ be compact and let $T\sub\R^{n}$
be compact with non-empty interior. Then there exists a (non-empty)
convex set ${\cal {\cal C}\sub{\cal B}}_{+}^{n}$ of optimal regular
Borel covering measures of $K$ by $T$. That is, for every $\mu\in{\cal {\cal C}}$
we have that $\mu*\one_{T}\ge\one_{K}$ and 
\[
N^{*} (K,T )=\int_{\R^{n}}d\mu
\]
\end{prop}
\begin{proof}
Since by Theorem \ref{thm:StrongDuality} $N^{*} (K,T )=N_{\omega} (K,T )$,
we may take a sequence $\left(\nu_{k}\right)\sub{\cal D}_{+}^{n}$
of discrete covering measures of $K$ by $T$ satisfying 
\[
\int_{\R^{n}}d\nu_{k}\underset{k\to\infty}{\longrightarrow}N^{*} (K,T ).
\]
Since the unit ball of the space of regular Borel measures is sequentially
compact in the weak{*} topology, there exists a subsequence $\left(\nu_{k_{l}}\right)$
and a regular Borel measure $\mu$ such that $\nu_{k_{l}}\overset{w^{*}}{\longrightarrow}\mu$
(this is the well-known Banach-Alaoglu's theorem). Let us show that
$\mu*\one_{T}\ge\one_{K}$. Indeed, let $x\in K$ and let $f\ge\one_{T}$
be a compactly supported continuous function. Then 
\[
1\le\left(\nu_{k_{l}}*f\right)\left(x\right)=\int f\left(x-y\right)d\nu_{k_{l}}\left(y\right)\underset{l\to\infty}{\longrightarrow}\int f\left(x-y\right)d\mu=\left(\mu*f\right)\left(x\right).
\]
Taking a monotone sequence $\left(f_{k}\right)$ of compactly supported
continuous functions satisfying $f_{k}\ge\one_{T}$ and point-wise
converging to $\one_{T}$, it follows by the monotone convergence
theorem that $\left(\mu*\one_{T}\right)\left(x\right)\ge1$, as needed.
Since the covering condition $\mu*\one_{T}\ge\one_{K}$ is preserved
under convex combinations, as is the total measure, it follows that
the set of optimal covering measures of $K$ by $T$ is convex.\end{proof}

\begin{rem}
\label{rem:OptimalCovSep}One might be tempted to ask whether there
exists a measure which is simultaneously optimal-separating and optimal-covering,
this turns out to be, in general, not correct. Indeed, one may consider
the following example. Let $T$ be the cross polytope in $\R^{3}$,
that is, $\conv(\pm e_{1},\pm e_{2},\pm e_{3})$, and let $K=\conv(e_{1},e_{2},e_{3})$
(where $\conv\left(A\right)$ denotes the convex hull of $A$). That
is, $K$ is a two dimensional triangle in $\R^{3}$. Clearly, $N(K,T)=N^{*}(K,T)=1$.
However, if there existed a measure $\mu$ which was both optimal-separating
and optimal-covering then in particular it would have had to be supported
in $K$, therefore we would get that the weighted covering of $K$
by the central section of $T$ with the plane $(1,1,1)^{\perp}$ is
also $1$. This section, which can also be written as $L=\conv(\left(e_{i}-e_{j}\right)/2\,:i,j=1,2,3)$,
is the hexagon $\frac{K-K}{2}$. We claim, however, that $N^{*}(K,L)>1$.
Indeed, the vertex $e_{1}$, for example, is covered by the copies
of $L$ centered at the triangle $\conv(e_{1},\frac{e_{1}+e_{2}}{2},\frac{e_{1}+e_{3}}{2})=\Delta_{1}$
and similarly define $\Delta_{2},\Delta_{3}$. By the assumption of
covering, $\mu(\Delta_{i})\ge 1$. On the other hand, if it were true
that $\mu(K)=1$ we would get, for example, that 
\[
\mu(\frac{e_{1}+e_{2}}{2})=\mu(\Delta_{1}\cap\Delta_{2})=\mu(\Delta_{1})+\mu(\Delta_{2})-\mu(\Delta_{1}\cup\Delta_{2})\ge 2-1=1.
\]
As this would also apply to $\frac{e_{1}+e_{3}}{2},\frac{e_{2}+e_{3}}{2}$,
it is a contradiction. Note that this argument actually shows that
$N^{*}(K,L)=\frac{3}{2}$ and further that the only optimal weighted
covering of $K$ by $L$ is given by the measure $\frac{1}{2}\delta_{\frac{e_{1}+e_{3}}{2}}+\frac{1}{2}\delta_{\frac{e_{2}+e_{3}}{2}}+\frac{1}{2}\delta_{\frac{e_{1}+e_{2}}{2}}$.
 Moreover, note that $K$ and $L$ satisfy that $N_{\omega} (K,L )\neq N (K,L )$,
hence providing a simple example for the fact that classical and weighted
covering numbers are not equal in general. By Proposition \ref{prop:WeakDuality},
$K$ and $L$ also provide a simple example for the fact that classical
covering and separation numbers are not equal in general. 
\end{rem}

\subsection{A Glivenko-Cantelli class}\label{sec:GCC}

In this section our goal is somewhat technical. We wish to use a uniform
measure to bound $N_{\omega} (K,T )$, however it is not
a member of ${\cal D}_{+}^{n}$. We claim that if we find some uniform
covering measure of a set $K$ by a convex set $T$ (supported on
some compact Borel set) with total mass $m$, then $N_{\omega} (K,T )\le m$.
This is because uniform measures can be approximated well by discrete
ones, and requires a proof. To this end, we need to recall the definition
of a Glivenko-Cantelli class. Let $\xi_{1},\xi_{2},\dots$ be a sequence
of i.i.d $\R^{n}$-valued random vectors having common distribution
$P$. The empirical measure $P_{k}$ is formed by placing mass $1/k$
at each of the points $\xi_{1}\left(\omega\right),\xi_{2}\left(\omega\right),\dots,\xi_{k}\left(\omega\right)$
. A class ${\cal A}$ of Borel subsets $A\in{\cal A}$ of $\R^{n}$
is said to be a Glivenko-Cantelli class for $P$ if 
\[
\sup_{A{\cal \in A}}\left|P_{n}\left(A\right)-P\left(A\right)\right|\overset{a.s.}{\longrightarrow}0
\]
In the following lemma, we will invoke a Glivenko-Cantelli theorem
for the class ${\cal C}_{n}$ of convex subsets of $\R^{n}$. Namely,
in \cite[Example 14]{EPS1979} it is shown that if a probability distribution
$P$ satisfies that $P (\partial K )=0$ for all $K\in{\cal C}_{n}$
then ${\cal C}_{n}$ is a Glivenko-Cantelli class for $P$. 
\begin{lem}
\label{lem:GCC}Let $K\sub\R^{n}$ and let $T\sub\R^{n}$ be a convex
set. Let $\mu$ be a uniform measure on some compact Borel set $A\sub\R^{n}$,
that is $d\mu=c\one_{A}dx$ for some $c>0$. Suppose that $\mu$ is
a covering measure of $K$ by $T$. Then 
\[
N_{\omega} (K,T )\le\mu (\R^{n} t).
\]
\end{lem}
\begin{proof}
Let $\eps>0$. We need to show that there exists a finite discrete
measure $\nu$ such that 
\[
\left(\nu*\one_{T}\right)\left(x\right)\ge1
\]
and $\nu\left(\R^{n}\right)\le\frac{1}{1-\eps}\mu\left(\R^{n}\right)$.
To this end, let $\mu_{0}=\frac{1}{c\vol\left(A\right)}\mu$ be the
uniform probability measure on $A$, let $\xi_{1},\xi_{2},\dots$
be a sequence of i.i.d $\R^{n}-$valued random vectors having common
distribution $\mu_{0}$, and let $\mu_{n}$ be the corresponding empirical
measure. The assumption that $\mu$ is a covering measure of $K$
by $T$ is equivalent to the condition that $\mu_{0}\left(x+T\right)\ge\frac{1}{c\vol\left(A\right)}$
for all $x\in K$. Since $\mu (\partial L )=0$ for all $L\in{\cal C}_{n}$,
it is implied by \cite[Example 14]{EPS1979} that ${\cal C}_{n}$
is a Glivenko-Cantelli class for $\mu_{0}$ and so, for some $k>1$,
\[
\sup_{L\in{\cal C}_{n}}\left|\mu_{0}\left(L\right)-\mu_{k}\left(L\right)\right|<\frac{\eps}{c\vol\left(A\right)}
\]
almost surely. In particular, there exists a discrete measure (one
of the $\mu_{k}$'s) $\nu_{0}=\sum_{i=1}^{k}\frac{1}{k}\delta_{x_{i}}$
for which 
\[
\left(\nu_{0}*\one_{T}\right)\left(x\right)=v_{0}\left(x+T\right)\ge\frac{1-\eps}{c\vol\left(A\right)}
\]
for all $x\in K$. Thus the measure $\nu=\frac{c\vol\left(A\right)}{1-\eps}\nu_{0}$
is a covering measure of $K$ by $T$ with $\nu\left(\R^{n}\right)=\frac{1}{1-\eps}\mu\left(\R^{n}\right)$,
as required. 
\end{proof}

\subsection{\label{sub:Volume-bounds}Volume bounds}

In this section we divide the proof Theorem \ref{thm:VolumeBounds}
into the following two propositions. 
\begin{prop}
\label{prop:volumeUpBound}Let $K\sub\R^{n}$ be compact and let $T\sub\R^{n}$
be compact with non-empty interior. Then 
\[
N^{*}(K,T)\le\frac{\vol\left(K-T\right)}{\vol\left(T\right)}.
\]
Additionally, if $T$ is convex then
\[
N_{\omega}(K,T)\le\frac{\vol\left(K-T\right)}{\vol\left(T\right)}.
\]
\end{prop}
\begin{proof}
By Theorem \ref{thm:StrongDuality}, it suffices to prove that $M^{*}(K,T)\le\frac{\vol\left(K+T\right)}{\vol\left(T\right)}$.
Let $\mu\in{\cal B}_{+}^{n}$ be a $T$-separated measure, that is,
$\one_{T}*\mu\le1$. Then
\begin{align*}
\int_{K}\vol\left(T\right)d\mu\left(x\right) & =\int_{K}d\mu\left(x\right)\int_{K+T}\one_{T}\left(y-x\right)dy=\int_{K+T}dy\int_{K}\one_{T}\left(y-x\right)d\mu\left(x\right)\\
 & \le\int_{K+T}\left(\one_{T}*\mu\right)\left(y\right)dy\le\int_{K+T}dy=\vol\left(K+T\right)
\end{align*}
 and so $\vol (T )M^{*} (K,T )\le\vol (K+T )$
as claimed.

Alternatively, one may verify that the measure $\mu_{0}=\one_{K-T}\frac{dx}{\vol\left(T\right)}$
is a covering measure of $K$ by $T$, from which the claim also follows.
By Lemma \ref{lem:GCC}, the latter argument implies that 
\[
N_{\omega}\left(K,T\right)\le\frac{\vol\left(K-T\right)}{\vol\left(T\right)}.
\]
\end{proof}
\begin{prop}
Let $K\sub\R^{n}$ be compact and let $T\sub\R^{n}$ be compact with
non-empty interior. Then 
\[
\max\left\{ \frac{\vol\left(K\right)}{\vol\left(T\right)},1\right\} \le M_{\omega}(K,T).
\]
\end{prop}
\begin{proof}
By Theorem \ref{thm:StrongDuality}, it suffices to prove that $\max\left\{ \frac{\vol\left(K\right)}{\vol\left(T\right)},1\right\} \le N^{*}(K,T)$.
Let $\mu\in{\cal B}_{+}^{n}$ be a covering measure of $K$ by $T$,
that is $\mu*\one_{T}\ge\one_{K}$. Then 
\begin{align*}
\int\vol\left(T\right)d\mu\left(x\right) & =\int d\mu\left(x\right)\int_{\R^{n}}\one_{T}\left(y-x\right)dy=\int_{\R^{n}}dy\int\one_{T}\left(y-x\right)d\mu\left(x\right)\\
 & \ge\int\one_{K}\left(y\right)dy=\vol\left(K\right)
\end{align*}
and so $N^{*}\left(K,T\right)\ge\frac{\vol\left(K\right)}{\vol\left(T\right)}$.$ $
Moreover, let $x\in K$. Then 
\[
1\le\left(\mu*\one_{T}\right)\left(x\right)=\int_{\R^{n}}\one_{T}\left(x-y\right)d\mu\left(y\right)\le\int_{\R^{n}}d\mu
\]
 and so $N^{*}\left(K,T\right)\ge1$. 

Alternatively, one may verify that the measure $\mu_{0}=\one_{K}\frac{dx}{\vol\left(T\right)}$
is $T-$separated in $K$, from which the claim also follows. The
fact that $1\le N^{*}\left(K,T\right)$ also follows from 
\[
1\le M\left(K,T\right)\le M^{*}\left(K,T\right)=N^{*}\left(K,T\right).
\]

\end{proof}

\subsection{\label{sub:delta-equivalence}An equivalence between classical and
weighted covering}

In this section we prove Theorem \ref{thm:DeltaEquiv}
\begin{proof}
[Proof of Theorem \ref{thm:DeltaEquiv}]Fix $\delta>0$. Let $\left(x_{i},\omega_{i}\right)_{i\in I}$
be a finite weighted discrete covering of $K$ by $T_{1}$ with 
\[
\sum_{i\in I}\omega_{i}<N_{\omega}\left(K,T_{1}\right)+\eps.
\]
Without loss of generality we may assume that $\omega_{i}$ are rational
numbers and moreover, by allowing repetitions of the covering points,
we may assume that for all $i$, $\omega_{i}=\frac{1}{M}$ for some
arbitrarily large $M\in\bbN$. Denote $N=\left\lfloor N_{\omega}\left(K,T_{1}\right)\right\rfloor $
and let $0<\eps<1$ be small enough so that $N+1\le N_{\omega}\left(K,T_{1}\right)+\eps$.
$ $Our aim is to generate a classical covering of $K$ by $T_{1}+T_{2}$
from the above fractional covering by a random process, with cardinality
not larger than 
\[
\ln\left(4\overline{N}\left(K,T_{2}\right)\right)N_{\omega}\left(K,T_{1}\right)+\sqrt{\ln\left(4\overline{N}\left(K,T_{2}\right)\right)N_{\omega}\left(K,T_{1}\right)}.
\]
To this end, let $S$ be an integer to be determined later and let
$L>1$ be some real number also to be determined later. Each point
will be chosen independently with probability $p=\frac{S}{M}$. We
claim that with positive probability, for $S=\ln\left(4\overline{N}\left(K,T_{2}\right)\right)$
and $L=1+\frac{1}{\sqrt{S\left(N+1\right)}}$, the generated set is
a covering of $K$ by $T_{1}+T_{2}$ and at the same time the cardinality
of the generated set is not greater than 
\[
LS\left(N+1\right)\le LS\left(\sum\omega_{i}+1\right)\le LS\left(N_{\omega}\left(K,T_{1}\right)+\eps+1\right).
\]
First, we bound the probability that more than $LS\left(N+1\right)$
will turn out positive. Let $X_{i}$ denote the Bernoulli random variable
corresponding to $x_{i}$ and let $X$ denote their sum. Note that
there are at most $M\left(N+1\right)$ trials as $\sum_{i\in I}\frac{1}{M}<N_{\omega}\left(K,T_{1}\right)+\eps\le N+1$.
Denote the cardinality of $I$ by $\left|I\right|$. A standard Chernoff
bound tells us that this probability can be bounded as follows. For
any $t>0$ 
\begin{align*}
\PP\left(X\ge LS\left(N+1\right)\right) & =\PP\left(e^{Xt}\ge e^{LSNt}\right)\le\min_{t>0}\frac{\EE\left(e^{tX_{1}}\cdots e^{tX_{\left|I\right|}}\right)}{e^{LS\left(N+1\right)t}}\le\min_{t>0}\frac{\left[pe^{t}+\left(1-p\right)\right]^{M\left(N+1\right)}}{e^{LS\left(N+1\right)t}}\\
 & =\left(\frac{1}{L}\right)^{LS\left(N+1\right)}\left[\frac{\left(1-p\right)}{1-Lp}\right]^{\left(N+1\right)M\left(1-Lp\right)}\\
 & =\left(\frac{1}{L}\right)^{LS\left(N+1\right)}\left[1+\frac{S\left(L-1\right)}{M-LS}\right]^{\left(N+1\right)\left(M-LS\right)}\\
 & \simeq\left(\frac{e^{L-1}}{L^{L}}\right)^{S\left(N+1\right)}
\end{align*}
where at the third equality the minimum is attained at $e^{t}=L\cdot\frac{1-p}{1-Lp}$
and the last step holds for sufficiently large $M$ compared with
$S$. Set $L=1+\xi$, then for $0<\xi\le1$, one can verify that 
\[
\PP\left(X\ge LS\left(N+1\right)\right)\le\left(\frac{e^{L-1}}{L^{L}}\right)^{S\left(N+1\right)}\le e^{-S\left(N+1\right)\xi/3}.
\]

Next, we show that with sufficiently high probability our generated
set is a covering of $K$ by $T_{1}+T_{2}.$ To this end, pick a minimal
covering $\left\{ y_{i}\right\} \sub K$ (we insist the points of
the net belong to $K$) of $K$ by $ $$T_{2}$.  The cardinality
of such a minimal net is $\overline{N}\left(K,T_{2}\right)$. If every
point $y_{i}$ is covered by a translate $x_{j}+T_{1}$ then the whole
of $K$ is covered by the translates $x_{j}+T_{1}+T_{2}$ of our randomly
generated set, as we desire. Let us consider one specific point $y_{i}=y$
and check the probability that it is covered by our randomly generated
set. Since we insisted that $y\in K$ we know that 
\[
\sum_{\left\{ i\in I\,:\, y\in x_{i}+T_{1}\right\} }\frac{1}{M}\ge1
\]
which means that at least $M$ of the original translates $x_{i}+T$
include $y$. Therefore, the probability that $ $$y$ is not covered
is less than or equal to $\left(1-\frac{S}{M}\right)^{M}\le e^{-S}$.
Thus, the probability that one or more of the $T_{2}-$covering points
$\left\{ y_{i}\right\} $ is not covered is bounded from above by
$\overline{N}\left(K,T_{2}\right)e^{-S}$. 

To summarize the above, we bounded the probability that either $K$
is not covered or the generated set consists of more than $LS\left(N+1\right)$
points by 
\[
e^{-S\left(N+1\right)\xi/3}+\overline{N}\left(K,T_{2}\right)e^{-S}
\]
and so it is left to choose $S$ and $\xi$ so that this bound is
less than $1.$ As one can verify, the choices $\xi=\frac{1}{\sqrt{S\left(N+1\right)}}$
and $S=\ln\left(4\overline{N}\left(K,T_{2}\right)\right)$ satisfy
this requirement. Thus, $N (K,T_{1}+T_{2})$ is bounded
by 
\begin{align*}
LS\left(N+1\right) & =\left(1+\frac{1}{\sqrt{S\left(N+1\right)}}\right)\ln\left(4\overline{N}\left(K,T_{2}\right)\right)\left(N+1\right)\\
 & =\left(1+\frac{1}{\sqrt{\ln\left(4\overline{N}\left(K,T_{2}\right)\right)\left(N+1\right)}}\right)\ln\left(4\overline{N}\left(K,T_{2}\right)\right)\left(N+1\right)\\
 & \le\ln\left(4\overline{N}\left(K,T_{2}\right)\right)\left(N_{\omega}\left(K,T_{1}\right)+1\right)+\sqrt{\ln\left(4\overline{N}\left(K,T_{2}\right)\right)\left(N_{\omega}\left(K,T_{1}\right)+1\right)}
\end{align*}

\end{proof}

\subsection{\label{sec:The-metric-space-setting}The metric-space setting}

The notions of covering and separation make sense also in the metric
space setting. Let$\left(X,d\right)$ be a metric space (with the
induced metric topology), and $K\subset X$ some compact subset. We
shall denote the $\eps$-covering number of $K$ by 

\[
N (K,\eps )=\min\left\{ N\in\bbN\,:\,\exists x_{1},\dots x_{N}\in\R^{n};\,\, K\sub\bigcup_{i=1}^{N}B(x_{i},\eps)\right\} 
\]
where $B(x,\eps)=\{y\in X\,:\, d(x,y)\le\eps\}$. Similarly

\[
\overline{N} (K,\eps )=\min\left\{ N\in\bbN\,:\,\exists x_{1},\dots x_{N}\in K;\,\, K\sub\bigcup_{i=1}^{N}B(x_{i},\eps)\right\} 
\]

The corresponding notion of the separation number is defined to be
the maximal number of non-overlapping $\eps$-balls centered in $K$;
\[
M (K,\eps )=\max\left\{ M\in\bbN\,:\,\exists x_{1},\dots x_{M}\in K,\,\, B(x_{i},\eps)\cap B(x_{j},\eps)=\emptyset\,\,\forall i\neq j\right\} .
\]
In this case it makes sense also to define 
\[
\overline{M} (K,\eps )=\max\left\{ M\in\bbN\,:\,\exists x_{1},\dots x_{M}\in K,\,\, B(x_{i},\eps)\cap B(x_{j},\eps)\cap K=\emptyset\,\,\forall i\neq j\right\} ,
\]
and one should note that in the case $K=X$ these notions of course
coincide. Also note that the metric setting is inherently centrally
symmetric. However, since we no longer work in a linear space, some
of the arguments in the preceding sections need to be altered. 

Let us define weighted covering and separation in the metric setting,
and list the relevant theorems corresponding to those proved in previous
sections which hold in this setting. We shall remark only on the parts
of the proofs which are not identical to those from the linear realm.
\begin{defn}
Let $(X,d)$ be a metric space and $K\subset X$ compact. A sequence
of pairs $S=\{(x_{i}\,,\,\omega_{i}):\, x_{i}\in X,\,\omega_{i}\in\R^{+}\}_{i=1}^{N}$
with $N\in\bbN$ points and weights is said to be an weighted $\eps$-covering
of $K$ if for all $x\in K$, $\sum_{\{i:x\in B(x_{i}.\eps)\}}\omega_{i}\ge 1$.
The total weight of the covering is denoted by $\omega(S)=\sum_{i=1}^{N}\omega_{i}$.
The weighted $\eps$-covering number of $K$ is defined to be the
infimal total weight over all weighted $\eps$-coverings of $K$ and
is denoted by $N_{\omega} (K,\eps )$.
\end{defn}
Similarly, we may define (in a slightly different language) 
\[
\overline{N}_{\omega} (K,\eps )=\inf\left\{ \int_{X}d\nu\,:\,\forall x\,\,\int\one_{B(y,\eps)}(x)d\nu(y)\ge\one_{K}(x)\,,\nu\in{\cal D}_{+}(X)\,\,{\rm with}\,\,{\rm supp}\left(\nu\right)\sub K\right\} 
\]
where ${\cal D}_{+}(X)$ denotes all non-negative finite discrete
measures on $X$. Let ${\cal B}_{+}(X)$ denote all non-negative Borel
measures on $X$. The weighted covering number with respect to general
measures is defined by
\[
N^{*} (K,\eps )=\inf\left\{ \int_{X}d\mu\,:\,\forall x\,\,\int\one_{B(y,\eps)}(x)d\mu(y)\ge\one_{K}(x)\,,\mu\in{\cal B}_{+}(X)\right\} .
\]

The weighted notions of the separation number are defined similarly;
a measure $\rho$ is said to be $\eps$-separated if for all $x\in X$
, $\int\one_{B(y,\eps)}(x)d\rho(y)\le 1$ and $\eps$-separated in
$K$ if $\int\one_{B(y,\eps)}(x)d\rho(y)\le 1$ for all $x\in K$.
The weighted separation numbers, corresponding to $N_{\omega}\left(K,\eps\right)$,
$\overline{N}_{\omega}\left(K,\eps\right)$ and $N^{*}\left(K,\eps\right)$
are respectively defined by:

\[
M_{\omega}\left(K,\eps\right)=\sup\left\{ \int_{K}d\rho\,:\,\,\forall x\in X\,\int\one_{B(y,\eps)}(x)d\rho(y)\le 1,\,\,\rho\in{\cal D}_{+}(X)\right\} ,
\]
\[
\overline{M}_{\omega}\left(K,\eps\right)=\sup\left\{ \int_{K}d\rho\,:\,\,\forall x\in K\,\int\one_{B(y,\eps)}(x)d\rho(y)\le 1,\,\,\rho\in{\cal D}_{+}(X)\right\} 
\]
and
\[
M^{*}\left(K,\eps\right)=\sup\left\{ \int_{K}d\rho\,:\,\,\forall x\in X\,\int\one_{B(y,\eps)}(x)d\rho(y)\le 1,\,\,\rho\in{\cal B}_{+}(X)\right\} .
\]
\\

Our first result is a weak duality between weighted covering and separation
numbers;
\begin{thm}
\label{thm:MetricWeakDuality}Let $(X,d)$ be a metric space, $K\sub X$
compact and let $\eps>0$. Then 
\[
M_{\omega} (K,\eps )\le M^{*}  (K,\eps )\le N^{*} (K,\eps )\le N_{\omega} (K,\eps )
\]
\end{thm}
\begin{proof}
The first and last inequalities follow by definition, and so we should
only prove the center inequality. To this end let $\mu$ be a weighted
$\eps$-covering measure of $K$ and let $\rho$ be a weighted $\eps$-separated
measure. By our assumptions we have that $\int\one_{B(y,\eps)}(x)d\rho(y)\le1$
and $\int\one_{B(y,\eps)}(x)d\mu(y)\ge\one_{K}(x)$ for all $x\in X$.
Thus 
\begin{align*}
\int_{K}d\rho\left(x\right)= & \int\one_{K}\left(x\right)\cdot d\rho\left(x\right)\le\int\int\one_{B(y,\eps)}(x)d\mu(y)d\rho\left(x\right)\\
= & \int d\rho\left(x\right)\int d\mu\left(y\right)\one_{B(y,\eps)}(x)\\
= & \int d\mu\left(y\right)\int d\rho\left(x\right)\one_{B(x,\eps)}(y)\\
\le & \int d\mu\left(y\right)
\end{align*}
and so $M^{*}\left(K,\eps\right)\le N^{*}\left(K,\eps\right)$. Similarly,
one may show that $\overline{M}_{\omega}\left(K,\eps\right)\le\overline{N}_{\omega}\left(K,\eps\right)$.
\end{proof}
As a corollary of Theorem \ref{thm:MetricWeakDuality}, we immediately
get the following equivalence relation between the classical and weighted
covering numbers:
\begin{cor}
\label{cor:MetricHomothetyEq}Let $(X,d)$ be a metric space, $K\sub X$
compact and let $\eps>0$. Then
\[
N (K,2\eps )\le N_{\omega} (K,\eps )\le N(K,\eps )
\]
\end{cor}
\begin{proof}
By Theorem \ref{thm:MetricWeakDuality}, $M\left(K,\eps\right)\le M_{\omega}\left(K,\eps\right)\le N_{\omega}\left(K,\eps\right)\le N\left(K,\eps\right)$
and so we only need to verify the inequality $N\left(K,2\eps\right)\le M\left(K,\eps\right)$.
Indeed, let $\left(x_{i}\right)_{i=1}^{N}\sub K$ be $\eps$-separated.
Hence, for every $x\in K$ there exists some $i\in1,\dots,N$ such
that $B\left(x,\eps\right)\cap B\left(x_{i},\eps\right)\neq\emptyset$
which by the triangle inequality means that $x\in B\left(x_{i},2\eps\right)$.
Thus, $\left(x_{i}\right)_{i=1}^{N}$ is a $2\eps$-covering of $K$
and so $N (K,2\eps )\le M (K,\eps)$, as needed. 
\end{proof}

\section{\label{sec:Hadwiger}The Levi-Hadwiger problem}

In this section we prove Theorem \ref{thm:WeightedHadwiger} and Corollary
\ref{cor:Hadwiger}. To this end, we shall need some preliminary results,
and before that, some notation.

\subsection{Preliminary results}

Denote the Euclidean open ball of radius $r>0$ and centered at $x$
by $B\left(x,r\right)\sub\R^{n}$. For Denote the segment between
two vectors $x,y\in\R^{n}$ by $\left[x,y\right]=\left\{ \lambda x+\left(1-\lambda\right)y\,:\,0\le\lambda\le1\right\} $.
Let $\partial A$ denote the boundary of a set $A\sub\R^{n}$.

\subsubsection{A Homothetic intersection}

We will need the following lemma, the proof of which was kindly shown
to us by Rolf Schneider and is reproduced here.
\begin{lem}
\label{lem:CcapC}[Schneider] Let $K\sub\R^{n}$ be a centrally-symmetric
convex body. Let $a\in K$ and let $p$ be the intersection point
of $\partial K$ with the ray emanating from $0$ and passing through
a. Assume that $\left(a+K\right)\cap K$ is homothetic to $K$. Then
there exists a closed convex cone $C\sub\R^{n}$ (with vertex $\left\{ 0\right\} $)
such that $K=\left(p-C\right)\cap\left(C-p\right)$.\end{lem}
\begin{proof}
Denote the homothety $h$ defined by $hK=K\cap\left(K+a\right)$.
Since $K$ is centrally symmetric, it follows that $hK$ is symmetric
about $\frac{a}{2}$, and since $hK$ is homothetic to $K$ it follows
that $hK=\frac{a}{2}+\alpha K$, where $\alpha=\frac{p-a/2}{p}$.
Thus $hK=\alpha\left(K-p\right)+p$, which means that $p=hp$ is the
center of homothety of $h(x)=\alpha(x-p)+p$. 

Define the cone 
\[
C_{o}=\left\{ \lambda\left(p-y\right)\,:\, y\in{\rm int}\, K,\lambda\ge0\right\} 
\]
and denote its closure by $C$. Let us prove that $\left(p-C_{o}\right)\cap\left(C_{o}-p\right)\sub K$;
assume towards a contradiction that there exists $x\in\left(p-C_{o}\right)\cap\left(C_{o}-p\right)$
such that $x\not\in K$. Let $y,z\in K$ be the points for which 
\[
\left[p,y\right]=K\cap\left[p,x\right],\,\,\,\,\left[-p,z\right]=K\cap\left[-p,x\right]
\]
and Consider the quadrangle $T$ in $K$ with vertices $\pm p,y,z$.
\begin{figure}[h]
\begin{centering}
\includegraphics[width=9cm]{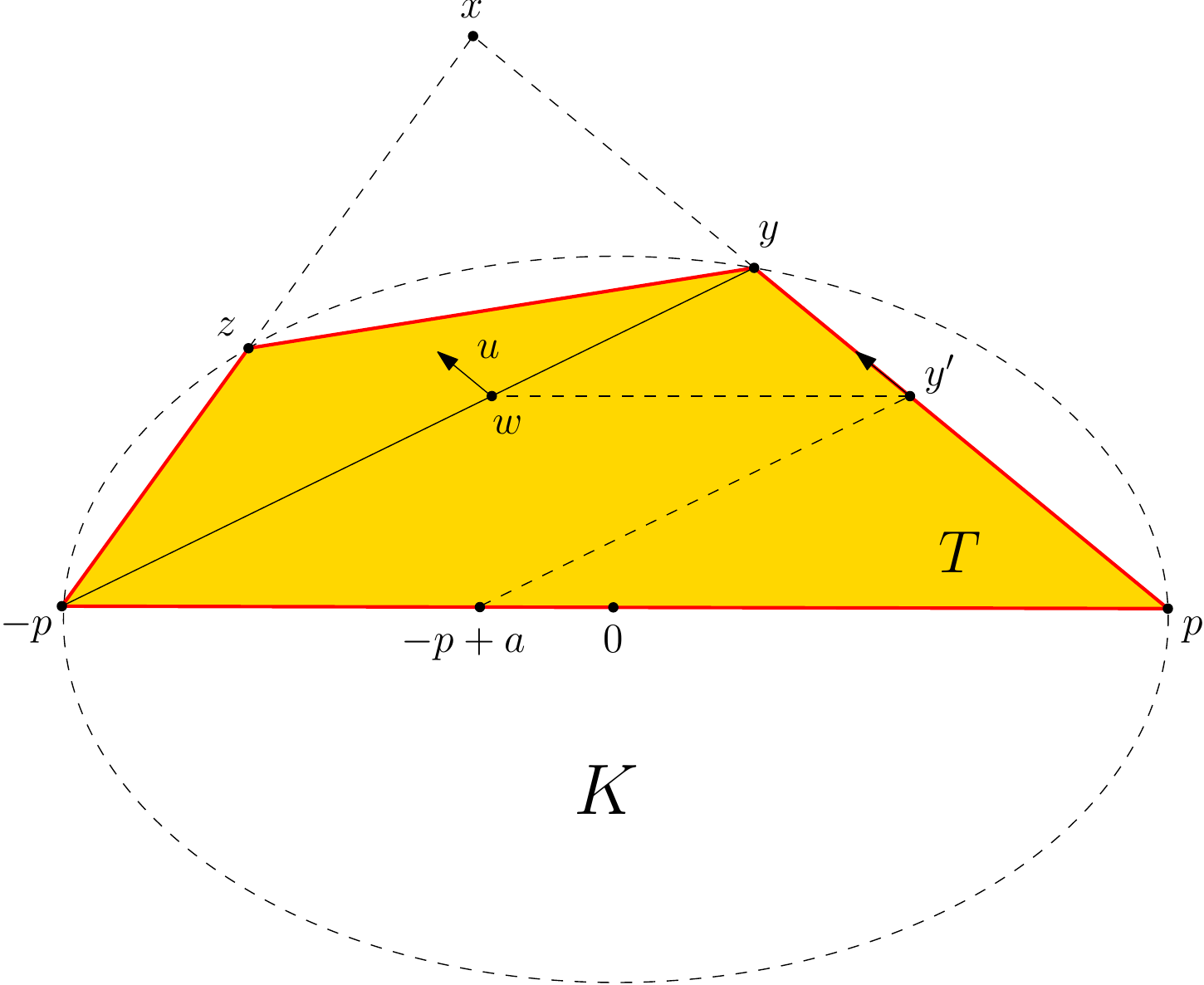}
\par\end{centering}

\caption{\label{fig:ConeCapCone}the vector $w+u$ belongs to $K$ and so $y'+u\in K\cap\left(K+a\right)$. }
\end{figure}
 Since $p$ is the center of homothety of $h$, the point $y'=\alpha\left(y-p\right)+p\in\left(hK\right)\cap\left[p,x\right]$
belongs to the boundary of $hK$. However, since the point $w=\alpha\left(y+p\right)-p$
is in the interior of $T$, it follows that, for some $\eps>0$, both
$w+u\in T$ and $y'+u\in T$, where $u=\eps\cdot\left(y-p\right)$
(see Figure \ref{fig:ConeCapCone}). Since $y'=w+a$, it follows that
$y'+u=\left(w+u\right)+a\in T+a$, and hence $y'+u\in K\cap\left(K+a\right)$,
a contradiction to the fact that $\left[p,y'\right]=\left(hK\right)\cap\left[p,x\right]$. 

We have proved that $\left(p-C_{o}\right)\cap\left(C_{o}-p\right)\sub K$
and hence $\left(p-C\right)\cap\left(C-p\right)\sub K$. The inclusion
$K\sub\left(p-C\right)\cap\left(C-p\right)$ trivially holds, and
thus $K=\left(p-C\right)\cap\left(C-p\right)$.
\end{proof}
We remark that if $K$ is not centrally symmetric, one may slightly
adjust Lemma \ref{lem:CcapC} and its proof in order to conclude the
following lemma. 
\begin{lem}
\label{lem:Non-Symm-CcapC}Let $K\sub\R^{n}$ be a convex body containing
the origin in its interior. Let $a\in K$ and assume that the intersection
point of $\partial K$ with the ray emanating from $0$ and passing
through a is an exposed point of $K$, denoted by $p$. Let $q$ denote
the point in $\partial K$ for which $0\in\left(q,p\right)$. Assume
that $\left(a+K\right)\cap K$ is homothetic to $K$. Then there exist
closed convex cones $C_{1},C_{2}\sub\R^{n}$ (both with vertex $\left\{ 0\right\} $)
such that $K=\left(p+C_{1}\right)\cap\left(q+C_{2}\right)$.
\end{lem}
The main difference between the proof of Lemma \ref{lem:CcapC} and
the proof of Lemma \ref{lem:Non-Symm-CcapC} is that, in the latter,
in order to prove that $p$ is the center of homothety of $h$, we
need to use the assumption that $p$ is an exposed point of $K$.
This is done by using the exact same argument as in the equality case
of Rogers-Shepard inequality in \cite{SchBook}. We shall not have
use of Lemma \ref{lem:Non-Symm-CcapC} in this note, and we omit the
proof's details.

\subsubsection{Covering a convex body by its interior}

It will be convenient to work with the weighted covering number of
a set $K$ by its interior ${\rm int}\left(K\right)$: $N_{\omega}\left(K,{\rm int}\left(K\right)\right)$.
The definition of this number is literally the same as for compact
sets;
\[
N_{\omega} (K,{\rm int} (K ) )=\inf\left\{ \nu(\R^{n})\,:\,\nu*\one_{{\rm int} (K )}\ge\one_{K}\,,\nu\in{\cal D}_{+}^{n}\right\} .
\]
We claim that covering a compact set, fractionally, by its interior
is the limit of fractionally covering it by infinitesimally smaller homothetic
copies of itself. More precisely, we prove the following.
\begin{lem}
\label{lem:FracCovInt} Let $K\sub\R^{n}$ be compact with non-empty
interior. Then 
\[
N_{\omega} (K,{\rm int}\left(K\right) )=\lim_{\lambda\to1^{-}}N_{\omega} (K,\lambda K ).
\]
 \end{lem}
\begin{proof}
Assume without loss of generality that $0\in\inte K$. The inequality
\[
N_{\omega}\left(K,{\rm int}\left(K\right)\right)\le\lim_{\lambda\to1^{-}}N_{\omega}\left(K,\lambda K\right)
\]
is straightforward by definition. For the opposite direction, let
$\mu=\sum_{i=1}^{N}\alpha_{i}\delta_{x_{i}}$ be a covering measure
of $K$ by ${\rm int}\left(K\right)$, i.e., $\mu*\one_{{\rm int}\left(K\right)}\ge\one_{K}.$
Note that if $x\in K$, $x\in\bigcap_{i\in A}\left(x_{i}+{\rm int}\left(K\right)\right)$
for some set of indices $A$, then 
\[
B (x,r )\sub\bigcap_{i\in A}\left(x_{i}+{\rm int}\left(K\right)\right)
\]
 for some open ball $B (x,r )$. Since $1\le\left(\mu*\one_{{\rm int}\left(K\right)}\right)\left(x\right)$,
it follows that for all $y\in B (x,r )$ we also have 
\[
1\le\left(\mu*\one_{{\rm int}\left(K\right)}\right)\left(y\right)=\sum_{i=1}^{N}\alpha_{i}\one_{x_{i}+{\rm int}\left(K\right)}\left(y\right).
\]
Hence, as $K$ is compact, there exists $\delta>0$ such that for
all $x\in K$,
\begin{align*}
1 & \le\sum_{i=1}^{N}\alpha_{i}\one_{x_{i}+{\rm int}\left(K\right)}\left(\left(1+\delta\right)x\right)=\sum_{i=1}^{N}\alpha_{i}\one_{\frac{1}{1+\delta}{\rm int}\left(K\right)}\left(x-\frac{x_{i}}{1+\delta}\right)=\left(\nu*\one_{\frac{1}{1+\delta}{\rm int}\left(K\right)}\right)\left(x\right)
\end{align*}
where $\nu=\sum_{i=1}^{N}\alpha_{i}\delta_{\frac{x_{i}}{1+\delta}}$.
Therefore, $1\le\mu*\one_{\frac{1}{1+\delta}{\rm int}\left(K\right)}\le\mu*\one_{\lambda_{0}K}$
for some $0<\lambda_{0}<1$, and so 
\[
\lim_{\lambda\to1^{-}}N (K,\lambda K )\le N_{\omega} (K,{\rm int}\left(K\right) ),
\]
from which the desired equality is implied.
\end{proof}

\subsubsection{Antipodal sets}

In this section we recall a beautiful result by Danzer and Gr\"{u}nbaum,
which we will need to invoke later on. To state their result, recall
that given a convex body $K\sub\R^{n}$, a set of points $A\sub K$
is said to be an antipodal set in $K$ if for each distinct pair of
points in $A$ there is a pair of distinct parallel supporting hyperplanes
of $K$, each containing one of the two points.

Danzer and Gr\"{u}nbaum \cite{DanzerGrunbaum} proved the following
theorem. 
\begin{thm}
\label{thm:DG}[Danzer and Gr\"{u}nbaum] The maximal cardinality
of an antipodal set in a convex body $K\sub\R^{n}$ is bounded from
above by $2^{n}$. Moreover, equality holds if and only if $K$ is
a parallelotope. 
\end{thm}

\subsection{Completing the proofs}

We turn to prove the weighted version of the Levi-Hadwiger problem. 
\begin{proof}
[Proof of Theorem \ref{thm:WeightedHadwiger}] Suppose first that
$K$ is not centrally symmetric. Then the volume inequality in Theorem
\ref{thm:VolumeBounds}, immediately implies that 
\[
\lim_{\lambda\to 1^{-}}
N(K,\lambda K)
\le
\lim_{\lambda\to 1^{-}}
\frac{\vol (K-\lambda K)}{\vol(\lambda K)}={2n \choose n},
\]
as required. Of course, in the symmetric case the same argument gives
the bound $2^{n}$. But we proceed differently so as to be able to
analyze the equality case. 

Suppose that $K$ is centrally symmetric. Without loss of generality,
we assume that $K$ has non-empty interior and that an open ball $B(0,r)$
of radius $r>0$ is contained in $K$. By Lemma \ref{lem:FracCovInt},
we may work with the weighted covering number of $K$ by its interior
$N_{\omega} (K,{\rm int}\left(K\right) )$,  and by Lemma
\ref{lem:GCC} we may also consider uniform covering measures to bound
$N_{\omega} (K,{\rm int}\left(K\right) )$ from above. Indeed,
consider the uniform measure $\mu$ on $K$ with density $\frac{2^{n}}{\vol(K)}$,
that is 
\begin{equation}
d\mu\left(y\right)=2^{n}\frac{\one_{K}\left(y\right)}{\vol\left(K\right)}dy.\label{eq:mu-def}
\end{equation}
Let us verify that $\mu$ is a covering measure of $K$ by ${\rm int}\left(K\right).$
Indeed, let $x\in K$. Then 
\begin{align*}
\left(\mu*\one_{{\rm int}\left(K\right)}\right)\left(x\right) & =\frac{2^{n}}{\vol(K)}\int\one_{{\rm int}\left(K\right)}\left(y\right)\one_{K}\left(x-y\right)dy=2^{n}\frac{\vol\left(K\cap\left(x+K\right)\right)}{\vol(K)}.
\end{align*}
Since 
\begin{equation}
K\cap\left(x+K\right)\supseteq\frac{K}{2}+\frac{1}{2}\left[K\cap\left(2x+K\right)\right]\supseteq\frac{K+x}{2},\label{eq:doubleInclusion}
\end{equation}
it follows that
\begin{equation}
2^{n}\frac{\vol\left(K\cap\left(x+K\right)\right)}{\vol(K)}\ge2^{n}\frac{\vol\left(K/2\right)}{\vol\left(K\right)}=1,\label{eq:ineq-vol-Kinteresect-K}
\end{equation}
as required.  This means that $N_{\omega}\left(K,{\rm int}\left(K\right)\right)\le\mu\left(\R^{n}\right)=2^{n}.$
To address the equality case, assume that for some centrally symmetric
convex body $K$ we have $N_{\omega}(K,{\rm int}(K))=2^{n}$. In particular,
for no $0<c<1$ is $c\mu$ (for $\mu$ given in \eqref{eq:mu-def})
a covering measure of $K$ by ${\rm int}(K)$. Therefore, the inequality
in \eqref{eq:ineq-vol-Kinteresect-K} must be an equality for some
$x\in K$. Indeed, if not, a standard compactness argument shows that
there exists $c\in(0,1)$ such that for all $x\in K$, 
\[
c2^{n}\frac{\vol\left(K\cap\left(x+K\right)\right)}{\vol(K)}\ge1
\]
which means that $c\mu$ is a covering measure of $K$ by $\inte K$,
a contradiction to the assumption $N_{\omega}\left(K,\inte K\right)=2^{n}$. 

Next, note that the inequality \eqref{eq:ineq-vol-Kinteresect-K}
is strict if and only if at least one of the inclusions in \eqref{eq:doubleInclusion}
is strict and, moreover, the rightmost inclusion in \eqref{eq:doubleInclusion}
is strict as long as $x\in K$ is not an extremal point of $K$. Thus,
the preceding two arguments imply that $K$ has at least one extremal
point $x_{0}\in K$ for which $\left(x_{0}+K\right)\cap K=\frac{K}{2}+\frac{1}{2}\left[K\cap\left(2x_{0}+K\right)\right]=\frac{K+x_{0}}{2}.$

Our aim for the remaining part of the proof, is to show that $K$
actually has at least $2^{n}$ extremal points $x_{1},\dots,x_{2^{n}}\in K$
such that $\left(x_{i}+K\right)\cap K=\frac{K+x_{i}}{2}$ for all
$i=1,\dots,2^{n}$, and use the characterization given in Lemma \ref{lem:CcapC}
for $K$ in order to deduce that $A=\left\{ x_{1},\dots,x_{2^{n}}\right\} $
is an antipodal set of $K$. Finally, we shall invoke Theorem \ref{thm:DG}
to conclude that $K$ is a parallelotope.

Assume that there exists exactly $k$ extremal points of $K$ $x_{1},\dots,x_{k}\in K$
such that 
\[
\left(x_{i}+K\right)\cap K=\frac{K}{2}+\frac{1}{2}\left[K\cap\left(2x_{i}+K\right)\right]=\frac{K+x_{i}}{2}
\]
for all $i=1,\dots,k$. Then, by using the same compactness argument
as before, it follows that there exists $0<c<1$ such that for all
$x\in K\setminus\left\{ B\left(x_{1},r\right),\dots,B\left(x_{k},r\right)\right\} $,
\[
\left(\left(c\mu\right)*\one_{{\rm int}\left(K\right)}\right)\left(x\right)=c\mu\left(x+\inte K\right)\ge1.
\]
Since $B\left(0,r\right)\sub\inte K$, we have that $B\left(x_{i},r\right)\sub x_{i}+\inte K,$
and so it follows that the measure 
\[
\nu=c\cdot\mu+\left(1-c\right)\sum_{i=1}^{k}\delta_{x_{i}}
\]
is a covering measure of $K$ by ${\rm int}\left(K\right)$. Therefore,
the equality assumption $N_{\omega}\left(K,\inte K\right)=2^{n}$
implies that $\nu\left(\R^{n}\right)=c2^{n}+\left(1-c\right)k\ge2^{n}$
which implies that $k\ge2^{n}.$ Concluding the above, there exist
at least $2^{n}$ extremal points $A=\left\{ x_{1},\dots,x_{2^{n}}\right\} $
in $K$ such that $K\cap\left(x_{i}+K\right)=\frac{K+x_{i}}{2}$ for
all $i\in A$. By Lemma \ref{lem:CcapC}, for each $i\in A$ there
exists a closed convex cone $C_{i}$ $ $such that $K=\left(x_{i}-C_{i}\right)\cap\left(C_{i}-x_{i}\right)$.

Let us next prove that if $x_{j}\neq x_{i}$ then $x_{j}$ belongs
to the boundary of $C_{i}-x_{i}$. Indeed, if $x_{j}$ belonged to
the interior of $C_{i}-x_{i}$ then it would have to belong to the
boundary of $x_{i}-C_{i}$ as it belongs to $\partial K$. However
since $x_{j}\neq x_{i}$. there exists a segment $\left(a,b\right)\sub x_{i}-C_{i}$
on the ray emanating from $x_{i}$ and passing through $x_{j}$ which
contains $x_{j}$. Together with the assumption that $x_{j}$ belongs
to the interior of $C_{i}-x_{i}$, it follows that $ $there exists
a segment $\left(a',b'\right)\sub\left(a,b\right)$ which both contains
$x_{j}$ and is contained in $K=\left(x_{i}-C\right)\cap\left(C_{i}-x_{i}\right)$,
a contradiction to the fact that $x_{j}$ is an extremal point of
$K$. 

It remains to show that $A$ is an antipodal set of $K$. Indeed,
since $x_{j}$ belongs to the boundary of $C_{i}-x_{i}$, the segment
$\left[-x_{i},x_{j}\right]$ is contained in the boundary of $C_{i}-x_{i}$
and so there exists a supporting hyperplane $H$ of $C_{i}-x_{i}$
which contains both $-x_{i}$ and $x_{j}$. In particular, $H$ supports
$K$. In other words, there exists a vector $v\in\R^{n}\setminus\left\{ 0\right\} $
such that for all $x\in C_{i}-x_{i}$, 
\[
\iprod xv\le\iprod{x_{j}}v=\iprod{-x_{i}}{,v}.
\]
Hence, for all $x\in x_{i}-C_{i}$,
\[
\iprod xv\ge\iprod{-x_{j}}v=\iprod{x_{i}}{,v},
\]
which means that 
\[
H'=H+\left(x_{i}-x_{j}\right)=\left\{ x+x_{i}-x_{j}\in\R^{n}\,:\,\iprod xv=\iprod{x_{j}}v\right\} =\left\{ y\in\R^{n}\,:\,\iprod yv=\iprod{x_{i}}v\right\} 
\]
contains $x_{i}$, supports $x_{i}-C_{i}$, and in particular supports
$K$. Thus, we conclude that$ $ $A$ is an antipodal set of $K$.
By Theorem \ref{thm:DG}, the maximal cardinality of an antipodal
set of a convex body is $2^{n}$ and equality holds only for parallelotopes,
and thus $K$ is a parallelotope.
\end{proof}

\begin{proof}
[Proof of Corollary \ref{cor:Hadwiger}]Fix $0<\delta<1$ and let
$n\ge3$. By Theorem \ref{thm:DeltaEquiv}, for any $0<\lambda<1$
\begin{align*}
N\left(K,\lambda K\right) & \le\ln\left(4\overline{N}\left(K,\delta\lambda K\right)\right)\left(N_{\omega}\left(K,\left(1-\delta\right)\lambda K\right)+1\right)\\
 & +\sqrt{\ln\left(4\overline{N}\left(K,\delta\lambda K\right)\right)\left(N_{\omega}\left(K,\left(1-\delta\right)\lambda K\right)+1\right)}\\
 & \le\ln\left(4\overline{N}\left(K,\delta\lambda K\right)\right)N_{\omega}\left(K,\left(1-\delta\right)\lambda K\right)+\sqrt{\ln\left(4\overline{N}\left(K,\delta\lambda K\right)\right)N_{\omega}\left(K,\left(1-\delta\right)\lambda K\right)}\\
 & +2\ln\left(4\overline{N}\left(K,\delta\lambda K\right)\right).
\end{align*}
 By Theorem \ref{thm:VolumeBounds}, we have that
\[
N_{\omega}\left(K,\left(1-\delta\right)\lambda K\right)\le\frac{\vol\left(K+\left(1-\delta\right)\lambda K\right)}{\vol\left(\left(1-\delta\right)\lambda K\right)}=\left(1+\frac{1}{\left(1-\delta\right)\lambda}\right)^{n}.
\]
By classical volume bounds we have that 
\[
\overline{N}\left(K,\delta\lambda K\right)\le M\left(K,\frac{\delta}{2}\lambda K\right)\le\frac{\vol\left(K+\frac{\delta}{2}\lambda K\right)}{\vol\left(\frac{\delta}{2}\lambda K\right)}=\left(1+\frac{2}{\lambda\delta}\right)^{n}
\]
and so 
\begin{align*}
N\left(K,\lambda K\right) & \le\left(1+\frac{1}{\left(1-\delta\right)\lambda}\right)^{n}\left[n\ln\left(4^{1/n}+\frac{2\cdot4^{1/n}}{\lambda\delta}\right)\right]\\
 & +\sqrt{\left(1+\frac{1}{\left(1-\delta\right)\lambda}\right)^{n}\left[n\ln\left(4^{1/n}+\frac{2\cdot4^{1/n}}{\lambda\delta}\right)\right]}+2n\ln\left(4^{1/n}+\frac{2\cdot4^{1/n}}{\lambda\delta}\right).
\end{align*}
Taking the limit $\lambda\to1^{-}$ implies that 
\begin{align*}
\lim_{\lambda\to1^{-}}N\left(K,\lambda K\right) & \le\left(1+\frac{1}{\left(1-\delta\right)}\right)^{n}\left[n\ln\left(4^{1/n}+\frac{2\cdot4^{1/n}}{\delta}\right)\right]\\
 & +\sqrt{\left(1+\frac{1}{\left(1-\delta\right)}\right)^{n}\left[n\ln\left(4^{1/n}+\frac{2\cdot4^{1/n}}{\delta}\right)\right]}+2n\ln\left(4^{1/n}+\frac{2\cdot4^{1/n}}{\delta}\right).
\end{align*}
By plugging $\delta=\frac{1}{n\ln\left(n\right)}$ we get
\begin{align*}
\lim_{\lambda\to1^{-}}N\left(K,\lambda K\right) & \le\left(2+\frac{1}{n\ln n-1}\right)^{n}\left[n\ln\left(4^{1/n}+2\cdot4^{1/n}n\ln n\right)\right]+\\
 & +\sqrt{\left(2+\frac{1}{n\ln n-1}\right)^{n}\left[n\ln\left(4^{1/n}+2\cdot4^{1/n}n\ln n\right)\right]}\\
 & +2n\ln\left(4^{1/n}+2\cdot4^{1/n}n\ln n\right).
\end{align*}
Since, for all $n\ge3$, $ $
\[
\left(2+\frac{1}{n\ln n-1}\right)^{n}\le2^{n}e^{1/\left(2\ln n-2/n\right)}\le2^{n}\left(1+\frac{1}{\ln n-1/n}\right)\le2^{n}\left(1+\frac{2}{\ln n}\right)
\]
and 
\[
n\ln\left(4^{1/n}+2\cdot4^{1/n}n\ln n\right)\le n\ln\left(4n\ln n\right)=n\ln4+n\ln n+n\ln\ln n,
\]
it follows that
\begin{align*}
\left(2+\frac{1}{n\ln n-1}\right)^{n}n\ln\left(4^{1/n}+2\cdot4^{1/n}n\ln n\right) & \le2^{n}\left(1+\frac{2}{n\ln n}\right)\left(n\ln n+n\ln\ln n+n\ln4\right)\\
 & \le2^{n}\left(n\ln n+n\ln\ln n+3.1n\right).
\end{align*}
Moreover, one may also check that
\begin{align*}
\sqrt{\left(2+\frac{1}{n\ln n-1}\right)^{n}n\ln\left(4^{1/n}+2\cdot4^{1/n}n\ln n\right)} & \le2^{n}0.5n
\end{align*}
and that
\[
2n\ln\left(4^{1/n}+2\cdot4^{1/n}n\ln n\right)\le2^{n}0.7n.
\]
Thus, it follows that
\begin{align*}
\lim_{\lambda\to1^{-}}N\left(K,\lambda K\right) & \le2^{n}\left(n\ln n+n\ln\ln n+5n\right).
\end{align*}

\end{proof}

\bibliographystyle{amsplain}
\addcontentsline{toc}{section}{\refname}\bibliography{Covering}

\providecommand{\bysame}{\leavevmode\hbox to3em{\hrulefill}\thinspace}
\providecommand{\MR}{\relax\ifhmode\unskip\space\fi MR }
\providecommand{\MRhref}[2]{%
  \href{http://www.ams.org/mathscinet-getitem?mr=#1}{#2}
}
\providecommand{\href}[2]{#2}
\begin{thebibliography}{10}

\bibitem{AKM2005}
S.~Artstein-Avidan, B.~Klartag, and V.~Milman, \emph{{The Santal\'{o} point of
  a function, and a functional form of the Santal\'{o} inequality}},
  Mathematika \textbf{51} (2004), no.~1-2, 33--48.

\bibitem{ArtsteinRaz2011}
S.~Artstein-Avidan and O.~Raz, \emph{{Weighted covering numbers of convex
  sets}}, Advances in Mathematics \textbf{227} (2011), no.~1, 730--744.

\bibitem{Ball88}
K.~Ball, \emph{{Logarithmically concave functions and sections of convex sets
  in {${\bf R}^n$}}}, Studia Math. \textbf{88} (1988), no.~1, 69--84.

\bibitem{Barth98}
F.~Barthe, \emph{{On a reverse form of the {B}rascamp-{L}ieb inequality}},
  Invent. Math. \textbf{134} (1998), no.~2, 335--361.

\bibitem{Barvinok2002}
A.~Barvinok, \emph{A course in convexity}, Graduate Studies in Mathematics,
  vol.~54, American Mathematical Society, Providence, RI, 2002. \MR{1940576
  (2003j:52001)}

\bibitem{Bezdek06}
K.~Bezdek, \emph{The illumination conjecture and its extensions}, Period. Math.
  Hungar. \textbf{53} (2006), no.~1-2, 59--69. \MR{2286460 (2007j:52018)}

\bibitem{BolGoh85}
{Boltjansky, V. G.} and I.~Gohberg, \emph{{Results and Problems in
  Combinatorial Geometry}}, Cambridge University Press, 1985.

\bibitem{BrassMoser05}
P.~Brass, W.~Moser, and J.~Pach, \emph{{Research problems in discrete
  geometry}}, Springer, New York, 2005.

\bibitem{DanzerGrunbaum}
L.~Danzer and B.~Gr\"{u}nbaum, \emph{{\"{U}ber zwei {P}robleme bez\"{u}glich
  konvexer {K}\"{o}rper von {P}. {E}rd{\H o}s und von {V}. {L}. {K}lee}}, Math.
  Z. \textbf{79} (1962), 95--99.

\bibitem{EPS1979}
J.~Elker, D.~Pollard, and W.~Stute, \emph{Glivenko-{C}antelli theorems for
  classes of convex sets}, Advances in Applied Probability \textbf{11} (1979),
  no.~4, pp. 820--833 (English).

\bibitem{ErdosRogers53}
P.~Erd\"{o}s and C.~A. Rogers, \emph{{The covering of {n}-dimensional space by
  spheres}}, J. London Math. Soc. \textbf{28} (1953), 287--293.

\bibitem{GohMar60}
I.~Gohberg and A.~Markus, \emph{{A problem on covering of convex figures by
  similar figures (in Russian)}}, Izv. Mold. Fil. Akad. Nauk SSSR \textbf{10}
  (1960), no.~76, 87--90.

\bibitem{Hadwiger57}
H.~Hadwiger, \emph{{Ungel\"{o}stes {P}robleme {N}r. 20}}, Elem. Math.
  \textbf{12} (1957), no.~6, 121.

\bibitem{Klartag2007}
B.~Klartag, \emph{Marginals of geometric inequalities}, Geometric Aspects of
  Functional Analysis (VitaliD. Milman and Gideon Schechtman, eds.), Lecture
  Notes in Mathematics, vol. 1910, Springer Berlin Heidelberg, 2007,
  pp.~133--166.

\bibitem{Klartag2005}
B.~Klartag and V.~D. Milman, \emph{{Geometry of log-concave functions and
  measures}}, Geometriae Dedicata \textbf{112} (2005), no.~1, 169--182.

\bibitem{Levi55}
F.~W. Levi, \emph{{\"{U}berdeckung eines {E}ibereiches durch
  {P}arallelverschiebung seines offenen {K}erns}}, Arch. Math. (Basel)
  \textbf{6} (1955), 369--370.

\bibitem{Lovasz1975}
L.~Lov\'{a}sz, \emph{{On the ratio of optimal integral and fractional covers}},
  Discrete Mathematics \textbf{13} (1975), no.~4, 383--390.

\bibitem{MartiniEt99}
H.~Martini and V.~Soltan, \emph{Combinatorial problems on the illumination of
  convex bodies}, Aequationes Math. \textbf{57} (1999), no.~2-3, 121--152.

\bibitem{Naszodi09}
M.~Nasz\'{o}di, \emph{{Fractional illumination of convex bodies}},
  Contributions to Discrete Mathematics \textbf{4} (2009), no.~2, 83--88.

\bibitem{RogersZong}
C.~A. Rogers and C.~Zong, \emph{Covering convex bodies by translates of convex
  bodies}, Mathematika \textbf{44} (1997), no.~1, 215--218.

\bibitem{SchBook}
R.~Schneider, \emph{Convex bodies: the {B}runn-{M}inkowski theory},
  Encyclopedia of Mathematics and its Applications, vol.~44, Cambridge
  University Press, Cambridge, 1993.

\end{thebibliography}

\bigskip

{\sc Tel Aviv University}

{\tt shiri@post.tau.ac.il}

{\tt boazslom@post.tau.ac.il}

\end{document}